\documentclass[pdflatex]{sn-jnl}
\usepackage{graphicx}%
\usepackage{multirow}%
\usepackage{amsmath,amssymb,amsfonts}%
\usepackage{amsthm}%
\usepackage{mathrsfs}%
\usepackage{quiver}
\usepackage[title]{appendix}%
\usepackage{xcolor}%
\usepackage{textcomp}%
\usepackage{manyfoot}%
\usepackage{booktabs}%
\usepackage{algorithm}%
\usepackage{algorithmicx}%
\usepackage{algpseudocode}%
\usepackage{listings}%

\usepackage[all]{xy}
\usepackage{bm}
\usepackage{enumerate}
\usepackage{amsrefs}
\newcommand{\Ker}{\operatorname{Ker}}

\theoremstyle{definition}
\newtheorem{dfn}{Definition}[section]
\theoremstyle{plain}
\newtheorem{thm}[dfn]{Theorem}
\newtheorem{prop}[dfn]{Proposition}
\newtheorem{lem}[dfn]{Lemma}
\newtheorem{cor}[dfn]{Corollary}
\newtheorem{conj}[dfn]{Conjecture}
\theoremstyle{remark}
\newtheorem{rem}[dfn]{Remark}

\newtheorem{constr}[dfn]{Construction}

\begin{document}

% \title[Article Title]{Article Title}


% \author*[1,2]{\fnm{First} \sur{Author}}\email{iauthor@gmail.com}



% \affil*[1]{\orgdiv{Department}, \orgname{Organization}, \orgaddress{\street{Street}, \city{City}, \postcode{100190}, \state{State}, \country{Country}}}

% \affil[2]{\orgdiv{Department}, \orgname{Organization}, \orgaddress{\street{Street}, \city{City}, \postcode{10587}, \state{State}, \country{Country}}}

% \affil[3]{\orgdiv{Department}, \orgname{Organization}, \orgaddress{\street{Street}, \city{City}, \postcode{610101}, \state{State}, \country{Country}}}

% \keywords{keyword1, Keyword2, Keyword3, Keyword4}



\title{Moduli spaces of rational curves on Artin-Mumford double solids}
\author*{\fnm{Fumiya} \sur{Okamura}}\email{fokamura941@g.chuo-u.ac.jp}
\affil{The Institute of Science and Engineering, Chuo University, 
1-13-27 Kasuga, Bunkyo-ku, Tokyo 112-8551, Japan}

\abstract{
    We describe the irreducible components of the moduli spaces of rational curves on Artin-Mumford double solids. 
    This provides the first example of Fano varieties that satisfy Geometric Manin's Conjecture with multiple Manin components in moduli space of rational curves for each degree.}

\keywords{Rational curves, Fano threefolds, moduli spaces, Brauer groups}
\pacs[MSC Classification]{14H10, 14J45, 14M20}

\maketitle
\section{Introduction}\label{section: introduction}
Throughout the paper, we work over an algebraically closed field $k$ of characteristic $0$. 
\textit{Fano varieties}, normal projective varieties with ample $\mathbb{Q}$-Cartier anticanonical divisors, are one of the main objects in algebraic geometry. 
Fano varieties have a rich geometry of \textit{rational curves} on them. 
Indeed, it is known that klt Fano varieties are \textit{rationally connected} (\cite{Mori1979}, \cite{MiyaokaMori1986}, \cite{KoMiMo1992c}, \cite{HaconMcKernan2007RC}, etc.), where Mori's \textit{bend-and-break} technique is used and developed in these papers. 

A classical question in algebraic geometry is to ask whether rationally connected varieties are \textit{rational}, \textit{stably-rational} or \textit{unirational}. 
This problem has been approached by various methods (e.g., \cite{IskovskikhManin1971BirAut}, \cite{ClemensGriffiths1972IJ}, \cite{ArtinMumford1972}, \cite{Voisin2015}). 
In this paper, we focus on the geometry of rational curves on \textit{Artin-Mumford's example} of unirational, but stably-irrational threefolds in \cite{ArtinMumford1972}. 
Such a threefold is realized as a double cover $X\rightarrow \mathbb{P}^{3}$ branched along a \textit{quartic symmetroid}, i.e., a quartic surface defined by the determinant of a symmetric matrix of linear forms. 
We call $X$ an \textit{Artin-Mumford double solid}. 
Artin-Mumford proved that its desingularization $\tilde{X}$ blown-up at ten nodes has $2$-torsion in its \textit{Brauer group} $\mathrm{Br}(\tilde{X})$, and that Brauer groups of smooth projective varieties are stably-birational invariant. 
Moreover, \cite{BlochSrinivas1983algcycle} and \cite{Voisin2006IntegralHodge} proved that for any smooth rationally connected projective threefold $V$, 
\[
    |\mathrm{Br}(V)| = |\Ker(B_{1}(V)_{\mathbb{Z}}\rightarrow N_{1}(V)_{\mathbb{Z}})|
\]
holds, where $B_{1}(V)_{\mathbb{Z}}$ (resp.\ $N_{1}(V)_{\mathbb{Z}}$) denotes the group of algebraically equivalent classes (resp.\ numerically equivalent classes) of $1$-cycles on $V$. 
In particular, this implies that $\tilde{X}$ has two algebraic classes in each numerical class of $1$-cycles. 

The aim of this paper is to prove the following:
\begin{thm}[= Theorem \ref{thm: AM space of lines} + Theorem \ref{thm: AM spaces of higher deg curves}]\label{thm: Main thm}
    Let $X$ be a general Artin-Mumford double solid and let $H$ be the ample generator of $\mathrm{Pic}(X)\cong \mathbb{Z}$. 
    \begin{enumerate}[(1)]
        \item The moduli space $\mathrm{Mor}(\mathbb{P}^{1},X,1)$ parametrizing $H$-lines consists of two irreducible components $M_{1}^{+}, M_{1}^{-}$ of the expected dimension $5$. 

        \item For each integer $d\ge 2$, the moduli space $\mathrm{Mor}(\mathbb{P}^{1},X,d)$ parametrizing rational curves of $H$-degree $d$ consists of four irreducible components $R_{d}^{+}, R_{d}^{-}, N_{d}^{+}, N_{d}^{-}$ of the expected dimension $2d + 3$, where $N_{d}^{+}$, $N_{d}^{-}$ parametrize $d$-sheeted covers of $H$-lines of $M_{1}^{+}$, $M_{1}^{-}$ respectively, and $R_{d}^{+}$, $R_{d}^{-}$ generically parametrize embedded, very free curves.  
    \end{enumerate}
    The same statements hold for the Kontsevich spaces $\overline{M}_{0,0}(X,d)$ parametrizing genus $0$ stable maps of $H$-degree $d\ge 1$, but the expected dimension of each component is less than by $3$. 
\end{thm}

This study is also motivated by \textit{Geometric Manin's Conjecture} and we apply the framework developed in \cite{HRS2004lowdegI}, \cite{CoskunStarr2009cubic}, \cite{LT2019Compos}, \cite{BLRT2022Fano3}, etc.. 
Geometric Manin's Conjecture asserts that for a sufficiently positive nef class $\beta \in N_{1}(X)_{\mathbb{Z}}$ on a Fano variety $X$, the number of irreducible components of $\mathrm{Mor}(\mathbb{P}^{1},X,\beta)$ which satisfy certain deformation properties is constant. 
Moreover, it is expected that the constant is equal to the size of the \textit{unramified Brauer group} $\mathrm{Br}_{\mathrm{nr}}(k(X)/k)$, which is isomorphic to the Brauer group of any smooth resolution of $X$. 
We give a precise statement of this conjecture in Subsection \ref{subsection: GMC}, or see \cite[Section 4.3]{LRT2026nonfree_section} for more generalized version of Geometric Manin's Conjecture. 

In \cite{LT2019Compos}, Lehmann-Tanimoto defined \textit{Manin components} as candidates of irreducible components which we should count in the conjecture, using the $a$-\textit{invariant} and the $b$-\textit{invariant}. 
Furthermore, they proposed a \textit{movable bend-and-break} conjecture, which asserts that any free rational curve $f\colon \mathbb{P}^{1}\rightarrow X$ on a Fano variety $X$ of sufficiently high $-K_{X}$-degree can deform into a union of two free rational curves. 
This is a powerful tool for running the induction argument for the proof of Geometric Manin's Conjecture. 

Fortunately, since Artin-Mumford double solids are terminal factorial del Pezzo threefolds of degree $2$, which we see in Section \ref{section: AM}, we can apply the results of \cite{Okamura2025Gt}. 
In particular, we can classify non-Manin components and prove the movable bend-and-break conjecture for our varieties. 
Thus, in this paper, we focus on studying the base case, i.e., the spaces of $H$-lines and $H$-conics. 
In particular, the geometry of lines on Artin-Mumford double solids is related to \textit{Reye congruences}, which are Enriques surfaces containing smooth rational curves (e.g., \cite{Cossec1983Reye}, \cite{DolgachevKondo2024EnriquesII}). 
We establish several properties of the spaces of lines on Artin-Mumford double solids by relating them to the geometry of Reye congruences in Subsection \ref{subsection: line vs Rey cong}. 

Another notable result is the following: 
\begin{thm}[= Theorem \ref{thm: AM strong MBB}]\label{thm: semi-main thm}
    Let $X$ be a general Artin-Mumford double solid. 
    For each integer $d\ge 2$, any component $M\subset \overline{M}_{0,0}(X,d)$ generically parametrizing birational stable maps contains unions of free curves $g\colon C_{1}\cup C_{2}\rightarrow X$ and $h\colon D_{1}\cup D_{2}\rightarrow X$ such that $g|_{C_{1}}\in M_{1}^{+}$ and $h|_{D_{1}}\in M_{1}^{-}$. 
\end{thm}
This is a stronger result than movable bend-and-break, and essential to prove Theorem \ref{thm: Main thm} as there are multiple components in the space $\overline{M}_{0,0}(X,1)$ of lines. 
The proof of Theorem \ref{thm: semi-main thm} involves the conic bundle structure of degree $2$ del Pezzo surfaces, treated in Subsection \ref{subsection: conics on dP2}. 

As a corollary of Theorem \ref{thm: Main thm}, we obtain Geometric Manin's Conjecture: 
\begin{cor}[= Corollary \ref{cor: AM GMC}]\label{cor: Cor of Main thm}
    Geometric Manin's Conjecture holds for general Artin-Mumford double solids.
\end{cor}
Geometric Manin's Conjecture has been proved for homogeneous varieties (\cite{Thomsen1998}, \cite{Kim2001}), smooth toric varieties (\cite{Bourqui2012}, \cite{Bourqui2016}), low degree general hypersurfaces (\cite{HRS2004lowdegI}, \cite{CoskunStarr2009cubic}, \cite{Beheshti2013}, \cite{Browning2017}, \cite{RiedlYang2019}, etc.), smooth Fano varieties of dimension at most $3$ (\cite{Testa2005}, \cite{Testa2009}, \cite{BLRT2022Fano3}, \cite{BurkeJovinelly2022Fano3}, etc.) or coindex at most $3$ (\cite{CoskunStarr2009cubic}, \cite{Castravet2004}, \cite{LT2019Compos}, \cite{ShimizuTanimoto2022dP1}, \cite{LT2021primeFano}, \cite{Okamura2024dP}, \cite{JovinellyOkamura2024coindex3}, etc.), terminal factorial del Pezzo threefolds of degree at least $3$ (\cite{Okamura2025Gt}), and so on.  
In these results, each Fano variety has a unique Manin component in the moduli space $\mathrm{Mor}(\mathbb{P}^{1},X,\beta)$ for each sufficiently positive $\beta\in N_{1}(X)_{\mathbb{Z}}$. 
Thus, to the best of the author's knowledge, Theorem \ref{thm: Main thm} (and Corollary \ref{cor: Cor of Main thm}) provides the first example of Fano varieties with multiple Manin components in the moduli space $\mathrm{Mor}(\mathbb{P}^{1},X,d)$ for each degree. 

\subsection*{Outline}\label{subsection: outline}
The paper is organized as follows. 
In Section \ref{section: preliminaries}, we collect preliminary results on the moduli spaces of rational curves. 
We also introduce Geometric Manin's Conjecture in Subsection \ref{subsection: GMC}. 
In Subsection \ref{subsection: conics on dP2}, we study conic bundles on del Pezzo surfaces of degree $2$, which is used to prove Theorem \ref{thm: semi-main thm}. 
In Section \ref{section: AM}, we construct Artin-Mumford double solids from linear systems of quadric surfaces in $\mathbb{P}^{3}$ (e.g., \cite{Beauville2016Luroth}). 
Then we see several basic properties of Artin-Mumford double solids (Subsection \ref{subsection: AM constr}) and the difference between algebraic and numerical equivalence classes of $1$-cycles (Subsection \ref{subsection: alg equiv vs num equiv}). 
In Section \ref{section: rational curves on AM}, we study the moduli spaces of rational curves on Artin-Mumford double solids. 
We concentrate on the geometry of lines and Reye congruences in Subsection \ref{subsection: line vs Rey cong}. 
Lastly, in Subsection \ref{subsection: higher deg rational curves on AM}, we prove the main theorems combining the results of \cite{Okamura2025Gt}. 

\section*{Acknowledgments}
The author would like to thank Sho Tanimoto for helpful discussions and continuous support. 
The author also would like to thank Shigeyuki Kondo for answering the questions on Enriques surfaces. 
The author is grateful to Brendan Hassett, Eric Jovinelly, Brian Lehmann, and Daniel Loughran for their valuable comments on an earlier draft of the paper. 
The author would like to thank the anonymous referee for the invaluable comments.

The author was partially supported by JST FOREST program Grant number JPMJFR212Z.

\section{Preliminaries}\label{section: preliminaries}
We begin by fixing our notation on numerical cones and moduli spaces of rational curves. 
Let $X$ be a projective variety. 
Let $N_{1}(X)_{\mathbb{Z}}$ and $N^{1}(X)_{\mathbb{Z}}$ be the set of numerical classes of $1$-cycles and Cartier divisors on $X$ respectively. 
Set $N_{1}(X) = N_{1}(X)_{\mathbb{Z}} \otimes \mathbb{R}$ and $N^{1}(X) = N^{1}(X)_{\mathbb{Z}} \otimes \mathbb{R}$. 
The dimension of the $\mathbb{R}$-vector spaces $N_{1}(X)$ and $N^{1}(X)$ is called the \textit{Picard number}, denoted by $\rho(X)$. 
We write the \textit{pseudo-effective cones} of $1$-cycles and divisors as $\overline{\mathrm{Eff}}_{1}(X) \subset N_{1}(X)$ and $\overline{\mathrm{Eff}}^{1}(X) \subset N^{1}(X)$. 
Also, let $\mathrm{Nef}_{1}(X) \subset N_{1}(X)$ and $\mathrm{Nef}^{1}(X) \subset N^{1}(X)$ stand for the \textit{nef cones} of $1$-cycles and divisors, which are the dual cones of $\overline{\mathrm{Eff}}^{1}(X)$ and $\overline{\mathrm{Eff}}_{1}(X)$ respectively with respect to the intersection pairing $N^{1}(X)\times N_{1}(X)\rightarrow \mathbb{R}$. 

A \textit{rational curve} on $X$ is a non-constant morphism $f\colon \mathbb{P}^{1}\rightarrow X$. 
For a numerical class $\alpha \in N_{1}(X)_{\mathbb{Z}}\setminus\{0\}$, let $\mathrm{Mor}(\mathbb{P}^{1},X,\alpha)$ be the moduli space parametrizing rational curves $f\colon \mathbb{P}^{1}\rightarrow X$ such that $f_{*}\mathbb{P}^{1} \equiv \alpha$, numerically equivalent. 
Let $\mathrm{Mor}(\mathbb{P}^{1},X)$ be the disjoint union of $\mathrm{Mor}(\mathbb{P}^{1},X,\alpha)$ for all $\alpha \in N_{1}(X)_{\mathbb{Z}}\setminus\{0\}$. 
We also frequently use the \textit{Kontsevich space} $\overline{M}_{0,r}(X,\alpha)$ parametrizing $r$-pointed stable maps $(f\colon C\rightarrow X, p_{1},\dots,p_{r})$ of genus $0$ with numerical class $f_{*}C \equiv \alpha$. 
We set $\overline{M}_{0,r}(X)$ to be the disjoint union of $\overline{M}_{0,r}(X,\alpha)$ for all $\alpha \in N_{1}(X)_{\mathbb{Z}}\setminus\{0\}$. 

We introduce the notion of \textit{free} and \textit{very free} rational curves. 
\begin{dfn}
    Let $X$ be a projective variety of dimension $n$. 
    A rational curve $f\colon \mathbb{P}^{1}\rightarrow X$ is \textit{free} (resp.\ \textit{very free}) if the following conditions hold: 
    \begin{itemize}
    \item the image $f(\mathbb{P}^{1})$ is contained in the smooth locus of $X$, and
    
    \item $H^{1}(X, f^{*}T_{X}(-1)) = 0$ (resp.\ $H^{1}(X, f^{*}T_{X}(-2)) = 0$).
    \end{itemize}
    The second condition is equivalent to the existence of non-negative (resp.\ positive) integers $a_{1}, \dots, a_{n}$ such that 
    \[f^{*}T_{X} \cong \mathcal{O}(a_{1}) \oplus \dots \oplus \mathcal{O}(a_{n}). \] 
\end{dfn}

Recall that Gorenstein terminal threefolds have at worst isolated \textit{compound Du Val singularities}, i.e., their general hyperplane sections are Du Val singularities \cite[Corollary 5.38]{KollarMori1998}. 
In particular, threefolds with only finitely many nodes (ordinary double points) are Gorenstein terminal. 
We collect several properties of (free) rational curves on Gorenstein terminal Fano threefolds as Artin-Mumford double solids are Gorenstein terminal Fano, furthermore, factorial by Lemma \ref{lem: AM factorial}. 

\begin{lem}[\cite{Kollar1996}, Theorem 1.3; \cite{LT2024dPfibI}, Lemma 2.2]\label{lem: Gt exp dim}
    Let $X$ be a Gorenstein terminal threefold. 
    Let $f\colon \mathbb{P}^{1}\rightarrow X$ be a rational curve on $X$. 
    Then we have 
    \[
        \dim_{[f]}\mathrm{Mor}(\mathbb{P}^{1},X)\ge -K_{X}\cdot f_{*}\mathbb{P}^{1} + \dim X. 
    \]
    Moreover, if $H^{1}(\mathbb{P}^{1}, f^{*}T_{X}) = 0$ (e.g., $f$ is free), then it is the equality and $[f] \in \mathrm{Mor}(\mathbb{P}^{1},X)$ is a smooth point. 
\end{lem}

\begin{lem}[cf.\ \cite{LT2024dPfibI}, Lemma 2.3]\label{lem: free locus}
    Let $X$ be a Gorenstein terminal threefold. 
    Let $M \subset \mathrm{Mor}(\mathbb{P}^{1}, X)$ be a dominant component, i.e., the evaluation map $s\colon \mathcal{C}\rightarrow X$ from the universal family $\pi\colon \mathcal{C}\rightarrow M$ is dominant. 
    Then there is an open dense subset $U\subset X$ such that any rational curve parametrized by $\pi(s^{-1}(U))$ is free. 
\end{lem}
\begin{proof}
    By \cite[Lemma 2.3]{LT2024dPfibI}, a general member of $M$ is free. 
    Let $\phi\colon \tilde{X}\rightarrow X$ be a smooth resolution and let $\tilde{M} \subset \mathrm{Mor}(\mathbb{P}^{1}, \tilde{X})$ be the component generically parametrizing strict transforms of general curves in $M$. 
    Then by \cite[Theorem 3.11]{Kollar1996}, there is an open dense subset $\tilde{U}\subset \tilde{X}$ such that any rational curve of $\tilde{M}$ which meets $\tilde{U}$ is free. 
    Let $V \subset X$ be the proper closed subset dominated by members of $M$ intersecting with the singular locus of $X$. 
    Since $\phi$ is isomorphic over $X\setminus V$, one can take $U$ to be $\phi(\tilde{U})\setminus V$. 
\end{proof}

\subsection{Geometric Manin's Conjecture}\label{subsection: GMC}
In this subsection, we introduce Geometric Manin's Conjecture. 
The paper \cite[Section 4.3]{LRT2026nonfree_section} extends Geometric Manin's Conjecture to more general settings, concerning the spaces of sections of Fano fibrations over smooth curves. 
However, in this paper, we do not treat this version of the conjecture. 

We first define the $a$-\textit{invariant} and the $b$-\textit{invariant}.
\begin{dfn}\label{dfn: a, b-inv}
    Let $X$ be a smooth projective variety with a nef and big divisor $L$. 
    Then the $a$-\textit{invariant} is defined by 
    \[
        a(X, L) = \inf\{t\in \mathbb{R} \mid K_{X} + tL \in \overline{\mathrm{Eff}}^{1}(X)\}. 
    \]
    The $b$-\textit{invariant} is defined by 
    \[
        b(X, L) = \dim F(X, L),
    \]
    where $F(X,L) := \{\alpha \in \mathrm{Nef}_{1}(X)\mid (K_{X} + a(X, L)L)\cdot \alpha = 0\}$. 
    When $X$ is singular, we define 
    \[a(X, L) := a(\tilde{X}, \phi^{*}L),\quad b(X,L) := b(\tilde{X}, \phi^{*}L), \]
    where $\phi\colon \tilde{X}\rightarrow X$ is a smooth resolution. 
    It is independent of the choice of smooth resolutions of $X$ by \cite[Proposition 2.7 and Proposition 2.10]{HTT2015Balanced}. 
\end{dfn}

\begin{rem}\label{rem: a, b-inv weak Fano}
    For a terminal $\mathbb{Q}$-factorial weak Fano variety $X$, we have $a(X, -K_{X}) = 1$ and $b(X, -K_{X}) = \rho(X)$. 
\end{rem}

The $a$-invariant is related to deformation properties on the spaces of rational curves: 
\begin{prop}[\cite{Okamura2025Gt}, Proposition 3.1; cf.\ \cite{LT2019Compos}, Proposition 4.2]\label{prop: non-dom vs higher a-inv}
    Let $X$ be a Gorenstein terminal Fano threefold. 
    Let $M\subset \mathrm{Mor}(\mathbb{P}^{1},X)$ be an irreducible component such that the evaluation map $s\colon \mathcal{C}\rightarrow X$ from the universal family $\pi\colon \mathcal{C}\rightarrow M$ is non-dominant. 
    Then the closure $Z$ of the image of $s$ satisfies 
    \[
        a(Z, -K_{X}|_{Z}) > a(X, -K_{X}) = 1. 
    \]
\end{prop}

\begin{prop}[\cite{Okamura2025Gt}, Proposition 4.1; cf.\ \cite{LT2019Compos}, Proposition 5.15]\label{prop: ev with reducible fib vs a-cov}
    Let $X$ be a Gorenstein terminal Fano threefold. 
    Let $M\subset \mathrm{Mor}(\mathbb{P}^{1},X)$ be an irreducible component such that the evaluation map $s\colon \mathcal{C}\rightarrow X$ from the universal family $\pi\colon \mathcal{C}\rightarrow M$ is dominant, but the general fiber is not irreducible. 
    Take a smooth resolution $\tilde{\mathcal{C}}\rightarrow \mathcal{C}$ and consider the composition $\tilde{s}\colon \tilde{\mathcal{C}}\rightarrow \mathcal{C}\rightarrow X$. 
    Then the finite part $f\colon Y\rightarrow X$ of the Stein factorization of $\tilde{s}$ satisfies 
    \[
        a(Y, -f^{*}K_{X}) = a(X, -K_{X}) = 1. 
    \]
\end{prop}

\begin{dfn}\label{dfn: a-cov face contracting}
    Let $X$ be a terminal $\mathbb{Q}$-factorial weak Fano variety. 
    \begin{itemize}
        \item We say that $f\colon Y\rightarrow X$ is an $a$-\textit{cover} if it is a dominant generically finite morphism of degree at least $2$ such that $a(Y, -f^{*}K_{X}) = 1$. 

        \item We say that an $a$-cover $f\colon Y\rightarrow X$ is \textit{face contracting} if the induced map $f_{*}\colon F(Y, -f^{*}K_{X})\rightarrow \mathrm{Nef}_{1}(X)$ is not injective. 
    \end{itemize}
\end{dfn}

In light of Propositions \ref{prop: non-dom vs higher a-inv} and \ref{prop: ev with reducible fib vs a-cov}, we define \textit{Manin components}. 
\begin{dfn}\label{dfn: Manin comp}
    Let $X$ be a terminal $\mathbb{Q}$-factorial weak Fano variety. 
    \begin{itemize}
        \item A morphism $f\colon Y\rightarrow X$ is a \textit{breaking morphism} if one of the following holds: 
        \begin{enumerate}[(1)]
            \item $f$ satisfies $a(Y, -f^{*}K_{X}) > 1$, 

            \item $f$ is an $a$-cover with Iitaka dimension $\kappa(K_{Y}-f^{*}K_{X}) > 0$, or 
            
            \item $f$ is face contracting. 
        \end{enumerate}

        \item An irreducible component $M \subset \mathrm{Mor}(\mathbb{P}^{1},X)$ is \textit{accumulating} if there exists a breaking morphism $f\colon Y\rightarrow X$ and an irreducible component $N\subset \mathrm{Mor}(\mathbb{P}^{1},Y)$ such that $f$ induces a dominant map $f_{*}\colon N\rightarrow M$. 

        \item An irreducible component $M\subset \mathrm{Mor}(\mathbb{P}^{1},X)$ is \textit{Manin} if it is not accumulating. 
    \end{itemize}
\end{dfn}

\begin{conj}[Geometric Manin's Conjecture, cf.\ \cite{Tanimoto2021GMCintro}]\label{conj: GMC}
    Let $X$ be a terminal $\mathbb{Q}$-factorial weak Fano variety. 
    Then there exists a nef $1$-cycle class $\alpha \in \mathrm{Nef}_{1}(X)_{\mathbb{Z}}$ such that for any $\beta \in \alpha + \mathrm{Nef}_{1}(X)_{\mathbb{Z}}$, the moduli space $\mathrm{Mor}(\mathbb{P}^{1},X,\beta)$ contains exactly $|\mathrm{Br}_{\mathrm{nr}}(k(X)/k)|$ Manin components. 
\end{conj}

\begin{rem}[cf.\ \cite{LRT2026nonfree_section}, Remark 4.13]\label{rem: nr Br}
    In Geometric Manin's Conjecture \ref{conj: GMC}, the \textit{unramified Brauer group} appears. 
    It is a stably-birational invariant, and when $X$ is a smooth projective variety, it is isomorphic to the \textit{Brauer group} $\mathrm{Br}(X) = H_{\text{\'{e}t}}^{2}(X, \mathbb{G}_{m})$ (e.g., \cite{CT-Skorobogatov2021}). 
    For a smooth rationally connected projective variety, the Brauer group is isomorphic to the torsion part of $H_{2}(X, \mathbb{Z})$. 
    Let $X$ be a smooth rationally connected projective threefold. 
    Firstly, \cite[Theorem 1]{BlochSrinivas1983algcycle} proves that algebraic equivalence and homological equivalence of $1$-cycles on $X$ coincide. 
    Secondly, \cite[Theorem 2]{Voisin2006IntegralHodge} proves the integral Hodge conjecture for $1$-cycles on $X$. 
    Therefore, we obtain the equality 
    \[
        |\mathrm{Br}(X)| = |\Ker(B_{1}(X)_{\mathbb{Z}}\rightarrow N_{1}(X)_{\mathbb{Z}})|, 
    \]
    where $B_{1}(X)_{\mathbb{Z}}$ denotes the set of algebraically equivalent classes of $1$-cycles. 
    It is a conjecture when $X$ is a smooth rationally connected projective variety of dimension greater than $3$. 
    Thus, Geometric Manin's Conjecture for a smooth weak Fano variety $X$ predicts the uniqueness of Manin component in $\mathrm{Mor}(\mathbb{P}^{1},X)$ parametrizing rational curves $f\colon \mathbb{P}^{1}\rightarrow X$ such that $f_{*}\mathbb{P}^{1} \sim_{\mathrm{alg}} \gamma$, algebraically equivalent, for each sufficiently positive $\gamma \in B_{1}(X)_{\mathbb{Z}}$. 
    As mentioned in the Introduction, Geometric Manin's Conjecture is known for various cases where the Brauer groups are trivial. 
\end{rem}

\subsection{Conic bundles on degree $2$ del Pezzo surfaces}\label{subsection: conics on dP2}
The aim of this subsection is to prove Lemma \ref{lem: conic bundles on dP2}, which is used to show Theorem \ref{thm: AM space of conics}. 

Let $S$ be a smooth del Pezzo surface of degree $2$. 
It is realized as the blow up of $\mathbb{P}^{2}$ at $7$ points $p_{1}, \dots, p_{7}$ in general position. 
For $i, j \in \{1,\dots, 7\}$ with $i\neq j$, we let 
\begin{itemize}
    \item $A_{i}$ be the exceptional curve over $p_{i}$, 

    \item $B_{ij}$ the strict transform of the line through $p_{i}, p_{j}$, 

    \item $C_{ij}$ the strict transform of the conic through all $p_{k}$ but $p_{i}, p_{j}$, and 

    \item $D_{i}$ the strict transform of the nodal cubic through all $p_{1}, \dots, p_{7}$ with node at $p_{i}$. 
\end{itemize}
These are the all $56 = 7 + 21 + 21 + 7$ lines on $S$. 
The surface $S$ admits a double cover $\pi\colon S\rightarrow \mathbb{P}^{2}$ defined by the anticanonical system $|-K_{S}|$. 
The Geiser involution on $S$ (i.e., the Galois involution over $\pi$) maps $A_{i}$ and $B_{ij}$ to $D_{i}$ and $C_{ij}$ respectively. 

One can classify conic bundles on $S$ in terms of singular fibers. 
Indeed, any conic bundle on $S$ has $6$ singular fibers, and the description of singular fibers is as follows: 
\begin{enumerate}[(I)]
    \item For each $i\in \{1,\cdots, 7\}$, the set of singular fibers is $\{A_{j} + B_{ij}\mid j\neq i\}$; 

    \item For each subset $\Lambda\subset \{1,\dots, 7\}$ with $|\Lambda| = 3$, the set of singular fibers is $\{A_{\lambda} + C_{\Lambda\setminus \{\lambda\}} \mid \lambda \in \Lambda \} \cup \{B_{\Gamma} + B_{\Delta} \mid \{1,\dots, 7\}\setminus \Lambda = \Gamma \sqcup \Delta, |\Gamma| = 2\}$, where $B_{\Gamma} := B_{ij}$ when $\Gamma = \{i,j\}\subset \{1,\dots, 7\}$, and so on; 

    \item For each $i,j\in \{1,\dots, 7\}$ with $i\neq j$, the set of singular fibers is $\{A_{j} + D_{i}\} \cup \{B_{ik} + C_{jk}\mid k\neq i,j\}$; 

    \item For each subset $\Lambda\subset \{1,\dots, 7\}$ with $|\Lambda| = 3$, the set of singular fibers is $\{B_{\Lambda\setminus \{\lambda\}} + D_{\lambda}\mid \lambda \in \Lambda \} \cup \{C_{\Gamma} + C_{\Delta} \mid \{1,\dots, 7\}\setminus \Lambda = \Gamma \sqcup \Delta, |\Gamma| = 2\}$; or

    \item For each $i\in \{1,\cdots, 7\}$, the set of singular fibers is $\{C_{ij} + D_{j}\mid j\neq i\}$. 
\end{enumerate}

Hence there are $126 = 7 + 35 + 42 + 35 + 7$ conic bundles. 

\begin{lem}\label{lem: conic bundles on dP2}
    Notation as above. 
    Let $\mathrm{sign}\colon F(S)\rightarrow \{1,-1\}$ be a map from the set of lines on $S$ which satisfies $\mathrm{sign}(A_{i}) \neq \mathrm{sign}(D_{i})$, $\mathrm{sign}(B_{ij}) \neq \mathrm{sign}(C_{ij})$ for any $i,j\in \{1, \dots, 7\}$ with $i\neq j$. 
    Suppose furthermore that for any conic bundle on $S$, if some singular fiber $\ell + \ell'$ satisfies $\mathrm{sign}(\ell) \neq \mathrm{sign}(\ell')$, then so does any singular fiber. 
    Then $S$ admits a conic bundle having singular fibers $\ell_{1} + \ell_{2}$ and $\ell'_{1} + \ell'_{2}$ such that $\mathrm{sign}(\ell_{1}) = \mathrm{sign}(\ell_{2}) = 1$ and $\mathrm{sign}(\ell'_{1}) = \mathrm{sign}(\ell'_{2}) = -1$. 
\end{lem}
\begin{proof}
    After reordering the indices, we may assume that there exists $ 0 \le m \le 7$ such that $\mathrm{sign}(A_{i}) = 1$ if and only if $i\le m$. 
    We prove the claim dividing into three cases. 
    \begin{enumerate}[(1)]
        \item First, we suppose $m = 7$. 
        Then $\mathrm{sign}(A_{i}) = 1$ for any $i\in \{1,\dots, 7\}$. 
        Consider the conic bundles of type (III). 
        For each pair $i, j\in \{1,\dots, 7\}$ with $i\neq j$, the components of the fiber $A_{j} + D_{i}$ have signs $\mathrm{sign}(A_{j}) = 1, \mathrm{sign}(D_{i}) = -1$. 
        By the assumption, we have $\mathrm{sign}(B_{ik}) \neq \mathrm{sign}(C_{jk})$. 
        Hence all $B_{ij}$ have the same sign. 
        We then consider the conic bundle of type (II) for each $\Lambda \subset \{1,\dots, 7\}$. 
        Since the components of $B_{\Gamma} + B_{\Delta}$ have the same sign, we see that $\mathrm{sign}(C_{\Lambda\setminus \{\lambda\}}) = \mathrm{sign}(A_{\lambda}) = 1$ for any $\lambda \in \Lambda$. 
        Varying $\Lambda \subset \{1,\dots, 7\}$, we conclude that $\mathrm{sign}(B_{ij}) = -1$ for any $i\neq j$. 
        Hence the claim holds. 

        \item Suppose that $m = 6$. 
        Assume that the claim does not hold. 
        Consider the conic bundles of type (III) for $i = 7$ and $j\neq 7$. 
        Since $\mathrm{sign}(A_{j}) = \mathrm{sign}(D_{7}) = 1$, we have $\mathrm{sign}(B_{k7}) = \mathrm{sign}(C_{jk}) = 1$ for any $j\neq k\le 6$. 
        Hence we obtain that
        \[
        \mathrm{sign}(B_{\Gamma}) = 
        \begin{cases}
            1, & \text{if }\ 7 \in \Gamma\\
            -1, & \text{otherwise. } 
        \end{cases}
        \]
        Now we consider the conic bundle of type (II) for $\Lambda = \{1,2,3\}$. 
        The components of the fiber $A_{1} + C_{23}$ have the same sign $1$, while the components of the fiber $B_{45} + B_{67}$ have different signs, a contradiction. 

        \item Finally, suppose that $2\le m\le 5$. 
        Assume that the claim does not hold. 
        Consider the conic bundles of type (III). 
        For the pair $(i,j) = (6,1)$, since the components of the fiber $A_{1} + D_{6}$ have the same sign $1$, so do the components of the fiber $B_{67} + C_{17}$ (the case $k = 7$). 
        On the other hand, for the pair $(i,j) = (7,2)$, since the components of the fiber $A_{2} + D_{7}$ have the same sign $1$, so do the components of the fiber $B_{17} + C_{12}$ (the case $k = 1$). 
        However, these imply $\mathrm{sign}(C_{17}) = 1 = \mathrm{sign}(B_{17})$, a contradiction. 
    \end{enumerate}
    Since one can apply the arguments in (1), (2) to the map $-\mathrm{sign}$, the claim holds for $m= 0,1$ as well. 
\end{proof}

\section{Artin-Mumford double solids}\label{section: AM}
\subsection{Construction of AM double solids}\label{subsection: AM constr}
Firstly, we construct Artin-Mumford double solids: 
\begin{constr}[e.g., \cite{Beauville2016Luroth}, Section 6.3]\label{constr: AM modern}
    For an integer $r$, let $S_{r} \subset |\mathcal{O}_{\mathbb{P}^{3}}(2)|\cong \mathbb{P}^{9}$ be the locus parametrizing quadrics of rank at most $r$ in $\mathbb{P}^{3}$. 
    Then there is a stratification
    \[
        \emptyset = S_{0} \subset S_{1} \subset S_{2} \subset S_{3} \subset S_{4} = |\mathcal{O}_{\mathbb{P}^{3}}(2)|. 
    \]
    Moreover, we see that 
    \begin{itemize}
        \item $S_{3}$ is a hypersurface in $S_{4}$ of degree $4$ whose singular locus is $S_{2}$, 
        
        \item $S_{2}$ is an irreducible subvariety of dimension $6$ and degree $10$, and 

        \item $S_{1}$ is an irreducible subvariety of dimension $3$ and degree $8$.
    \end{itemize}

    Let $\mathbb{G}(1,\mathbb{P}^{3})$ denote the Grassmannian of lines in $\mathbb{P}^{3}$. 
    Consider the incidence correspondence 
    \[
        \Sigma := \{(\ell, Q) \in \mathbb{G}(1,\mathbb{P}^{3}) \times |\mathcal{O}_{\mathbb{P}^{3}}(2)| \mid \ell \subset Q\}
    \]
    with the projections 
% https://q.uiver.app/#q=WzAsMyxbMCwwLCJcXFNpZ21hIl0sWzEsMCwifFxcbWF0aGNhbHtPfV97XFxtYXRoYmJ7UH1eezN9fSgyKXwuIl0sWzAsMSwiXFxtYXRoYmJ7R30oMSxcXG1hdGhiYntQfV57M30pIl0sWzAsMiwicCIsMl0sWzAsMSwicSJdXQ==
\[\begin{tikzcd}
	\Sigma & {|\mathcal{O}_{\mathbb{P}^{3}}(2)|.} \\
	{\mathbb{G}(1,\mathbb{P}^{3})}
	\arrow["q", from=1-1, to=1-2]
	\arrow["p"', from=1-1, to=2-1]
\end{tikzcd}\]

    Take the Stein factorization $q \colon \Sigma\xrightarrow{g} T_{4} \xrightarrow{h} S_{4}$. 
    Then $h\colon T_{4}\rightarrow S_{4}$ is a double cover branched along $S_{3}$. 

    Now, we take a general linear system $W \subset |\mathcal{O}_{\mathbb{P}^{3}}(2)|$ of dimension $3$, so that 
    \begin{itemize}
        \item $W$ is base point free, and
    
        \item any line on $\mathbb{P}^{3}$ which lies in the singular locus of some quadric $Q\in W$ is not contained in any other quadric of $W$. 
    \end{itemize}
    In this situation, the following also hold: 
    \begin{itemize}
        \item $\mathcal{D}_{W} := S_{3}\cap W$ is a nodal quartic surface, 
    
        \item $\mathrm{Sing}(\mathcal{D}_{W}) = S_{2}\cap W = \{q_{1}, \dots, q_{10}\}$ consists of ten nodes, and
    
        \item $S_{1}\cap W = \emptyset$.
    \end{itemize}
    Let $f\colon X:= h^{-1}(W)\rightarrow W$ be the restriction of $h$, which is the double cover branched along the quartic surface $\mathcal{D}_{W}$. 
    Then $X$ has ten nodes $p_{1}, \dots, p_{10}$ as singularity.
    We will call $X$ an \textit{AM double solid}. 
\end{constr}

Let $\phi \colon \tilde{X}\rightarrow X$ be the blow up at the nodes $p_{1}, \dots, p_{10}$, and let $\psi\colon \tilde{W}\rightarrow W$ be the blow up at $q_{1}, \dots, q_{10}$. 
Then we obtain a commutative diagram
% https://q.uiver.app/#q=WzAsNCxbMCwwLCJcXHRpbGRle1h9Il0sWzEsMCwiWCJdLFswLDEsIlxcdGlsZGV7V30iXSxbMSwxLCJXIl0sWzAsMiwiXFx0aWxkZXtmfSIsMl0sWzAsMSwiXFxwaGkiXSxbMSwzLCJmIl0sWzIsMywiXFxwc2kiLDJdXQ==
\[\begin{tikzcd}
	{\tilde{X}} & X \\
	{\tilde{W}} & W
	\arrow["\phi", from=1-1, to=1-2]
	\arrow["{\tilde{f}}"', from=1-1, to=2-1]
	\arrow["f", from=1-2, to=2-2]
	\arrow["\psi"', from=2-1, to=2-2]
\end{tikzcd}\]
where $\tilde{f}$ is the double cover branched along the strict transform $\tilde{\mathcal{D}}_{W}$ of the quartic surface $\mathcal{D}_{W}\subset W$. 
Let $\iota\colon X\rightarrow X$ (resp.\ $\tilde{\iota}\colon \tilde{X}\rightarrow \tilde{X}$) be the Galois involution over $f$ (resp.\ $\tilde{f}$). 

\begin{thm}[\cite{ArtinMumford1972}]\label{thm: AM irrational}
    Notation as above. 
    Then the Brauer group of $\tilde{X}$ is non-trivial: 
    \[\mathrm{Br}(\tilde{X})\cong \mathbb{Z}/2\mathbb{Z}. \]
    Moreover, $\tilde{X}$ (and hence $X$) is unirational but not stably-rational.  
\end{thm}

\begin{rem}\label{rem: Br grp vs torsion alg classes}
    As discussed in Remark \ref{rem: nr Br}, $\tilde{X}$ contains two algebraically equivalent classes in each numerical class of $1$-cycles. 
    In Subsection \ref{subsection: alg equiv vs num equiv}, we describe the difference between algebraic and numerical classes on $\tilde{X}$. 
\end{rem}

\begin{lem}\label{lem: AM factorial}
    An AM double solid $X$ is factorial.
\end{lem}
\begin{proof}
    Since $X$ is a Gorenstein terminal threefold, the $\mathbb{Q}$-factoriality and factoriality are equivalent by \cite[Lemma 5.1]{Kawamata1988}. 
    Hence it suffices to show $X$ is $\mathbb{Q}$-factorial.
    
    By \cite[Proposition 2.5]{HosonoTakagi2015symmetroid}, the variety $T_{4}$ as in Construction \ref{constr: AM modern} is $\mathbb{Q}$-factorial. 
    One can take a sequence $T_{4} =: X^{9}\supset X^{8}\supset \dots \supset X^{4}\supset X^{3}:= X$, where $X^{i}$ is an ample divisor on $X^{i+1}$ such that $\mathrm{Sing}(X^{i}) = X^{i}\cap \mathrm{Sing}(X^{i+1})$. 
    Applying \cite[Theorem 1]{RavindraSrinivas2006Lefschetz} repeatedly, we see that $X$ remains $\mathbb{Q}$-factorial.

    Alternatively, the $\mathbb{Q}$-factoriality of $X$ also follows from \cite{Clemens1983} and \cite[Lemma 3.5]{Endrass1999}.
\end{proof}
Lemma \ref{lem: AM factorial} implies that the blow-up threefold $\tilde{X}$ of an AM double solid $X$ at ten nodes has Picard rank $11$ and all of the $10$ exceptional divisors $E_{1},\dots, E_{10}$ are of type E3 in the sense of \cite[Theorem 1.32]{KollarMori1998} and \cite{Cutkosky1988}, i.e., $(E_{i}, -K_{\tilde{X}}|_{E_{i}}) \cong (\mathbb{P}^{1}\times \mathbb{P}^{1}, \mathcal{O}(1,1))$.

\begin{rem}\label{rem: AM not weak Fano}
    The smooth threefold $\tilde{X}$ is not weak Fano.
    Indeed, 
    \[(-K_{\tilde{X}})^{3} = (-\phi^{*}K_{X})^{3} - \sum_{i=1}^{10}E_{i}^{3} = 16 - 20 = -4 < 0. \] 
    Hence we consider Geometric Manin's Conjecture for AM double solids $X$, instead for $\tilde{X}$. 
\end{rem}

\subsection{Algebraic equivalence and numerical equivalence}\label{subsection: alg equiv vs num equiv}
We use the same notation as in Subsection \ref{subsection: AM constr}. 

For $\tilde{X}$, we have 
\[
    \mathrm{Pic}(\tilde{X}) \cong \mathbb{Z}\tilde{H} \oplus \sum_{i=1}^{10}\mathbb{Z}E_{i},
\]
where $\tilde{H}$ is the pull back of the fundamental divisor $H$ of $X$ and each $E_{i}\cong \mathbb{P}^{1}\times \mathbb{P}^{1}$ is the exceptional divisor over $p_{i}$. 

Let $N_{1}(\tilde{X})$ be the set of numerical classes of $1$-cycles on $\tilde{X}$. 
One can write $N_{1}(\tilde{X})$ using the dual basis as 
\[
    N_{1}(\tilde{X}) = \mathbb{R}\tilde{\ell} \oplus \sum_{i = 1}^{10}\mathbb{R}e_{i}, 
\]
where $\tilde{\ell}$ is the numerical class of strict transforms of $H$-lines which do not pass through any node of $X$, and $e_{i}$ is the numerical class of the push forward of lines on $\mathbb{P}^{1}\times \mathbb{P}^{1}$ via the inclusion $\mathbb{P}^{1}\times \mathbb{P}^{1}\cong E_{i}\hookrightarrow \tilde{X}$. 
For $i\neq j$, let $\tilde{\ell}_{i} = \tilde{\ell} - e_{i}$ (resp.\ $\tilde{\ell}_{i,j} = \tilde{\ell} - e_{i} - e_{j}$) be the numerical class of strict transforms of $H$-lines through $p_{i}$ (resp.\ $p_{i}$ and $p_{j}$). 

\begin{lem}\label{lem: AM alg equiv ve num equiv}
    Notation as above.
    \begin{enumerate}
        \item For any $i$, the generators $e_{i}^{+}$ and $e_{i}^{-} = \tilde{\iota}(e_{i}^{+})$ of the rulings of $E_{i}\cong \mathbb{P}^{1}\times \mathbb{P}^{1}$ are in the same numerical class $[e_{i}]$ but not algebraically equivalent.

        \item For any $i \neq j$, the two irreducible curves $\tilde{\ell}_{i,j}^{+}$ and $\tilde{\ell}_{i,j}^{-} = \tilde{\iota}(\tilde{\ell}_{i,j}^{+})$ in the numerical class $[\tilde{\ell}_{i,j}]$ are not algebraically equivalent. 

        \item For any $i$, the space $\overline{M}_{0,0}(\tilde{X}, \tilde{\ell}_{i})$ consists of a disjoint union of two elliptic curves $\tilde{\Gamma}_{i}^{+}$ and $\tilde{\Gamma}_{i}^{-}$.

        \item Let $\tilde{\ell}^{+} \in [\tilde{\ell}]$ be an irreducible curve. 
        Then $\tilde{\ell}^{+}$ and its conjugate $\tilde{\ell}^{-} = \tilde{\iota}(\tilde{\ell}^{+})$ are not algebraically equivalent.
    \end{enumerate}
\end{lem}
\begin{proof}
    For each $i\in \{1,\dots, 10\}$, the projection from $q_{i}\in W\cong \mathbb{P}^{3}$ induces a dominant morphism $\nu_{i}\colon \tilde{X}\rightarrow \mathbb{P}^{2}$ such that the restriction of $\nu_{i}$ to the complement of $\bigcup_{j\neq i}E_{j}$ is a conic bundle. 
    By the second condition for the linear web $W$, we see that the quartic surface $\mathcal{D}_{W}$ does not contain any lines through the nodes. 
    Since the singular locus of $\mathcal{D}_{W}$ consists of ten nodes, by \cite[Theorem 1.2]{Ottemetal2015quarticspectrahedra} (cf. \cite[p.139]{Cayley1869/71quartic}), the discriminant locus of $\nu_{i}$ is the union of two elliptic curves $\delta_{i}^{+} \cup \delta_{i}^{-}$. 
    Then, by \cite[pp.81--83]{ArtinMumford1972}, for each point $q^{*}\in \delta_{i}^{*}\setminus (\delta_{i}^{+}\cap \delta_{i}^{-})\ (*\in \{+,-\})$, the components $\tilde{\ell}_{i}^{*}$ and $\tilde{\iota}(\tilde{\ell}_{i}^{*})$ of $\nu_{i}^{-1}(q^{*})$ are algebraically equivalent, but $\tilde{\ell}_{i}^{+}$ and $\tilde{\ell}_{i}^{-}$ are not.

    Fix a point $q_{j} := \nu_{i}(E_{j}) \in \delta_{i}^{+}\cap \delta_{i}^{-}$ for $j\in \{1,\dots,10\}$ with $j\neq i$. 
    Then there are two ways to specialize the one parameter families of $1$-cycles, i.e., 
    \[
        \tilde{\ell}_{i}^{+}\rightarrow \tilde{\ell}_{i,j}^{+} + e_{j}^{+},\quad \mbox{as}\quad \delta_{i}^{+}\ni q^{+}\rightarrow q_{j}
    \]
    and 
    \[
        \tilde{\ell}_{i}^{-}\rightarrow \tilde{\ell}_{i,j}^{+} + e_{j}^{-},\quad \mbox{as}\quad \delta_{i}^{-}\ni q^{-}\rightarrow q_{j}, 
    \]
    after replacing signs $+,-$, components $\tilde{\ell}_{i}^{+}$, $\tilde{\iota}(\tilde{\ell}_{i}^{+})$, and $\tilde{\ell}_{i}^{-}$, $\tilde{\iota}(\tilde{\ell}_{i}^{-})$ if necessary. 
    Combining the fact that $\tilde{\ell}_{i}^{+}$ and $\tilde{\ell}_{i}^{-}$ are not algebraically equivalent, we conclude that $e_{j}^{+}$ and $e_{j}^{-}$ are not algebraically equivalent, hence nor are $\tilde{\ell}_{i,j}^{+}$ and $\tilde{\ell}_{i,j}^{-}$, proving (1) and (2). 
    
    As for (3), by the discussion above, one can take the curves $\tilde{\Gamma}_{i}^{+}$ and $\tilde{\Gamma}_{i}^{-}$ to be \'{e}tale double covers of the strict transforms of $\delta_{i}^{+}$ and $\delta_{i}^{-}$ via the blow up of $\mathbb{P}^{2}$ at the intersection points $\delta_{i}^{+}\cap \delta_{i}^{-}$, i.e., one can take $\tilde{\Gamma}_{i}^{+} = D_{1}'$ and $\tilde{\Gamma}_{i}^{-} = D_{2}'$ in \cite[p.81]{ArtinMumford1972}. 
    
    The statement (4) follows from the fact $\tilde{\ell} = \tilde{\ell}_{i,j} + e_{i} + e_{j}$ for any $i\neq j$.
\end{proof}

\begin{cor}\label{cor: AM lines through nodes}
    The fiber of the evaluation map $\overline{M}_{0,1}(X, \ell)\rightarrow X$ over a node $p_{i}$ is isomorphic to a connected union of two elliptic curves $\Gamma_{i}^{+}$ and $\Gamma_{i}^{-}$. 
    In particular, the algebraic equivalence and the numerical equivalence of $1$-cycles on $X$ coincide. 
\end{cor}
\begin{proof}
    One can see that $\Gamma_{i}^{+}$ and $\Gamma_{i}^{-}$ are the images of $\tilde{\Gamma}_{i}^{+}$ and $\tilde{\Gamma}_{i}^{-}$ by the morphism $\phi_{*}\colon \overline{M}_{0,0}(\tilde{X}, \tilde{\ell}_{i})\rightarrow \overline{M}_{0,0}(X,\ell)$ induced by $\phi\colon \tilde{X}\rightarrow X$. 
    This morphism maps $\tilde{\Gamma}_{i}^{+}$ and $\tilde{\Gamma}_{i}^{-}$ isomorphically to $\Gamma_{i}^{+}$ and $\Gamma_{i}^{-}$ respectively, but $\Gamma_{i}^{+}$ and $\Gamma_{i}^{-}$ intersect at the points corresponding to the lines through $p_{i}$ and another node. 
    This proves the first claim. 
    
    By Theorem \ref{thm: AM irrational}, Remark \ref{rem: nr Br}, and Lemma \ref{lem: AM alg equiv ve num equiv}, the unique nonzero element of $\Ker(B_{1}(\tilde{X})_{\mathbb{Z}}\rightarrow N_{1}(\tilde{X})_{\mathbb{Z}})$ is represented by a $1$-cycle $\tilde{\ell}_{i}^{+} - \tilde{\ell}_{i}^{-}$, where $\tilde{\ell}_{i}^{+} \in \tilde{\Gamma}_{i}^{+}$ and $\tilde{\ell}_{i}^{-} \in \tilde{\Gamma}_{i}^{-}$. 
    By the first claim, the push forwards $\phi_{*}\tilde{\ell}_{i}^{+}$ and $\phi_{*}\tilde{\ell}_{i}^{-}$ are algebraic equivalent, which concludes the second claim. 
\end{proof}

\section{Rational curves on AM double solids}\label{section: rational curves on AM}
\subsection{Lines on AM double solids and Reye congruences}\label{subsection: line vs Rey cong}
In this subsection, we discuss relations between the spaces of lines on AM double solids, the spaces of bitangent lines, and Reye congruences. 
We use the notation as in Section \ref{section: AM} freely. 

Let $\ell$ be a line on an AM double solid $X$, i.e., a curve with $H\cdot \ell = 1$. 
Then the image of $\ell$ by $f\colon X\rightarrow W \cong \mathbb{P}^{3}$ is a line by the projection formula. 
Conversely, if $L\subset W$ is a line, then the inverse image is a union of two lines if and only if $L$ is bitangent to the quartic surface $\mathcal{D}_{W}\subset W$: otherwise, the inverse image is an irreducible curve of $H$-degree $2$.
Therefore, $f$ induces a finite morphism $f_{*}\colon \overline{M}_{0,0}(X,1)\rightarrow \mathrm{Bit}(\mathcal{D}_{W})$ of degree $2$, where $\mathrm{Bit}(\mathcal{D}_{W})\subset \mathbb{G}(1,W)$ denotes the space of lines on $W$ bitangent to $\mathcal{D}_{W}$. 

Now, we define the \textit{Reye congruence} associated with $W$: 
\begin{dfn}\label{dfn: Reye congruence}
    For a linear web $W\subset |\mathcal{O}_{\mathbb{P}^{3}}(2)|$ as in Construction \ref{constr: AM modern}, we define the \textit{Reye congruence} associated with $W$ by 
    \[
        \mathrm{Rey}(W) := \{\mathfrak{l} \in \mathbb{G}(1,\mathbb{P}^{3})\mid \mathfrak{l}\subset \mathrm{Bs}(L) \text{ for some linear pencil }L\in \mathbb{G}(1,W)\}, 
    \]
    where $\mathrm{Bs}(L)$ denotes the base locus of a pencil $L\subset W$.
    We call an element $\mathfrak{l}\in \mathrm{Rey}(W)$ a \textit{Reye line}.
\end{dfn}
Note that for a Reye line $\mathfrak{l}$, the pencil $L$ whose base locus contains $\mathfrak{l}$ is uniquely determined since $W$ is base point free. 

From now on, we assume moreover that $W$ is \textit{excellent} in the sense of \cite{Cossec1983Reye}, or \cite[Definition 7.2.8]{DolgachevKondo2024EnriquesII}. It is equivalent to the condition that $\mathcal{D}_{W}\subset W$ does not contain lines. 
A general linear web $W$ satisfies this condition as well. 

\begin{thm}[e.g., \cite{DolgachevKondo2024EnriquesII}, Section 7.4]\label{thm: Reye Bit normalization}
    Assume $W$ is excellent. 
    Then the Reye congruence $\mathrm{Rey}(W)$ is an Enriques surface. 
    Moreover, we have the normalization morphism 
    \[
        \nu\colon \mathrm{Rey}(W)\rightarrow \mathrm{Bit}(\mathcal{D}_{W})
    \]
    which sends a Reye line $\mathfrak{l}\in \mathbb{G}(1,\mathbb{P}^{3})$ to the unique linear pencil $L\in \mathbb{G}(1,W)$ such that $\mathfrak{l}\subset \mathrm{Bs}(L)$. 
\end{thm}

Combining Theorem \ref{thm: Reye Bit normalization} and a small observation, we can describe the irreducible components of $\overline{M}_{0,0}(X, 1)$: 
\begin{thm}\label{thm: AM space of lines}
    Let $X$ be the AM double solid associated with an excellent web $W$. 
    Then the space $\overline{M}_{0,0}(X,1)$ of lines on $X$ consists of two irreducible components $M_{1}^{+}$ and $M_{1}^{-}$. 
    These components are conjugate by the involution $\iota_{*}$ induced by $\iota\colon X\rightarrow X$. 
    Moreover, there is a natural lift of the normalization map $\nu$ to $\overline{M}_{0,0}(X, 1)$: 
% https://q.uiver.app/#q=WzAsNCxbMCwxLCJcXG1hdGhybXtSZXl9KFcpIl0sWzIsMSwiXFxtYXRocm17Qml0fShcXG1hdGhjYWx7RH1fe1d9KS4iXSxbMiwwLCJcXG92ZXJsaW5le019X3swLDB9KFgsIDEpIl0sWzEsMCwiTV97MX1eeyt9Il0sWzMsMiwiIiwwLHsic3R5bGUiOnsidGFpbCI6eyJuYW1lIjoiaG9vayIsInNpZGUiOiJ0b3AifX19XSxbMiwxLCJmX3sqfSJdLFswLDEsIlxcbnUiXSxbMCwzLCJcXGJhcntcXG51fSJdXQ==
\[\begin{tikzcd}
	& {M_{1}^{+}} & {\overline{M}_{0,0}(X, 1)} \\
	{\mathrm{Rey}(W)} && {\mathrm{Bit}(\mathcal{D}_{W}).}
	\arrow[hook, from=1-2, to=1-3]
	\arrow["{f_{*}}", from=1-3, to=2-3]
	\arrow["{\bar{\nu}}", from=2-1, to=1-2]
	\arrow["\nu", from=2-1, to=2-3]
\end{tikzcd}\]
\end{thm}
\begin{proof}
    We prove the existence of the lift $\bar{\nu}$. 
    Let $\mathfrak{l} \in \mathrm{Rey}(W)$ be a Reye line. 
    By construction, any quadric $Q$ in the pencil $L := \nu(\mathfrak{l)}$ contains the Reye line $\mathfrak{l}$. 
    We consider the pairs $(\mathfrak{l}, Q)$ for $Q\in L$, which are nothing but the elements of $\Sigma_{W} := q^{-1}(W)$, where $q\colon \Sigma\rightarrow |\mathcal{O}_{\mathbb{P}^{3}}(2)|$ as in Construction \ref{constr: AM modern}. 
    Hence the morphism $g|_{\Sigma_{W}}\colon \Sigma_{W} \rightarrow X$ maps the set $\{(\mathfrak{l}, Q) \mid Q\in L\}$ to a line $\ell^{+}$ on $X$. 
    Therefore, we obtain a morphism $\bar{\nu}\colon \mathrm{Rey}(W)\rightarrow \overline{M}_{0,0}(X,1)$. 
    Since $\mathrm{Rey}(W)$ is irreducible and maps surjectively to $\mathrm{Bit}(\mathcal{D}_{W})$, we see that the image of $\bar{\nu}$ is an irreducible component $M_{1}^{+}\subset \overline{M}_{0,0}(X, 1)$. 
    Moreover, since any fiber of $f_{*}$ consists of two points which are conjugate by $\iota_{*}$ while $\nu$ is birational, we see that $\overline{M}_{0,0}(X, 1)$ consists of $M_{1}^{+}$ and its conjugate $M_{1}^{-} := \iota_{*}(M_{1}^{+})$, as desired. 
\end{proof}

\begin{rem}\label{rem: spaces of lines on AM non-normal}
    The map $\bar{\nu}\colon \mathrm{Rey}(W)\rightarrow M_{1}^{+}$ is not an isomorphism: let $L_{ij}\in \mathrm{Bit}(\mathcal{D}_{W})$ be the line passing through nodes $q_{i}, q_{j}$. 
    Then both of the components of $f^{-1}(L_{ij})$ are elements of $M_{1}^{+}$ while the preimage $\nu^{-1}(L_{ij})$ consists of four Reye lines. 
    Thus $M_{1}^{+}$ and $M_{1}^{-}$ are non-normal.
\end{rem}

The following lemmas are used to describe the spaces of conics on AM double solids (Theorem \ref{thm: AM space of conics}). 
\begin{lem}\label{lem: Reye reducible div}
    Let $\mathrm{Rey}(W)\subset \mathbb{G}(1, \mathbb{P}^{3})$ be the Reye congruence associated with an excellent web $W\subset |\mathcal{O}_{\mathbb{P}^{3}}(2)|$. 
    Let $\mathcal{O}_{\mathrm{Rey}(W)}(1)$ be the restriction of the very ample line bundle $\mathcal{O}_{\mathbb{P}^{5}}(1)$ via the Pl\"{u}cker embedding $\mathrm{Rey}(W)\hookrightarrow \mathbb{G}(1,\mathbb{P}^{3})\hookrightarrow \mathbb{P}^{5}$. 
    Then any non-integral divisor of $|\mathcal{O}_{\mathrm{Rey}(W)}(1)|$ contains a prime divisor $D$ with $\dim |D|\le 1$. 
    The same statement is true for the very ample linear system $|\omega_{\mathrm{Rey}(W)}(1)|$. 
\end{lem}
\begin{proof}
    First, we prove that for an effective divisor $D$ on $\mathrm{Rey}(W)$, $D$ satisfies $\dim |D| \le 1$ if and only if $D^{2} \le 2$ holds. 
    If $D$ is an effective divisor such that $D^{2} \le 0$, then clearly, $\dim |D| \le 1$. 
    Let $D$ be an effective divisor with $D^{2} > 0$. 
    By the Riemann-Roch theorem, we have 
    \[
        \chi(D) = \frac{1}{2}D\cdot (D - K_{\mathrm{Rey}(W)}) + \chi(0) = \frac{D^{2}}{2} + 1 
    \]
    since $\mathrm{Rey}(W)$ is an Enriques surface.
    By the Serre duality, we have $h^{2}(D) = h^{0}(K_{\mathrm{Rey}(W)} - D) \le h^{0}(-2D) = 0$. 
    Hence $h^{2}(D) = 0$. 
    Thus, we obtain that 
    \[
        \dim |D| = \frac{D^{2}}{2} + h^{1}(D) \ge 1.
    \]
    This implies $|D|$ dominates $\mathrm{Rey}(W)$, hence $D$ is nef.
    Moreover, $D$ is also big since $D^{2} > 0$. 
    Then the Kawamata-Viehweg vanishing theorem says $h^{1}(D) = 0$ since $D + K_{\mathrm{Rey}(W)}$ is also nef and big.
    Therefore, we obtain that $\dim |D| = D^{2}/2$, which implies the claim. 

    Assume that there exist non-integral divisors of $|\mathcal{O}_{\mathrm{Rey}(W)}(1)|$ which fail the lemma. 
    Since $\mathcal{O}_{\mathrm{Rey}(W)}(1)^{2} = 10$ (e.g., \cite[Theorem 7.4.4]{DolgachevKondo2024EnriquesII}), any such divisor must be the union of two prime divisors $D_{1} + D_{2}$ such that $D_{1}^{2} = D_{2}^{2} = 4$, $D_{1}\cdot D_{2} = 1$. 
    Now, consider the map $\phi := \phi_{|D_{1}|}\colon \mathrm{Rey}(W)\dashrightarrow \mathbb{P}^{2}$ defined by the linear system $|D_{1}|$. 
    Resolving the indeterminacies, we obtain the diagram 
% https://q.uiver.app/#q=WzAsMyxbMCwxLCJcXG1hdGhybXtSZXl9KFcpIl0sWzAsMCwiUyJdLFsxLDEsIlxcbWF0aGJie1B9XnsyfS4iXSxbMSwwLCJcXHBzaSIsMl0sWzEsMiwiXFx0aWxkZXtcXHBoaX0iXSxbMCwyLCJcXHBoaSIsMix7InN0eWxlIjp7ImJvZHkiOnsibmFtZSI6ImRhc2hlZCJ9fX1dXQ==
\[\begin{tikzcd}
	S \\
	{\mathrm{Rey}(W)} & {\mathbb{P}^{2}.}
	\arrow["\psi"', from=1-1, to=2-1]
	\arrow["{\tilde{\phi}}", from=1-1, to=2-2]
	\arrow["\phi"', dashed, from=2-1, to=2-2]
\end{tikzcd}\]
    By the projection formula, we have 
    \[
        1 = D_{1}\cdot D_{2} = \psi^{-1}_{*}D_{1}\cdot \psi^{*}D_{2} = \tilde{\phi}^{*}\mathcal{O}_{\mathbb{P}^{2}}(1)\cdot \psi^{*}D_{2} = \mathcal{O}_{\mathbb{P}^{2}}(1)\cdot \tilde{\phi}_{*}\psi^{*}D_{2}. 
    \]
    Hence, $\phi_{*}(D_{2})$ is contained in a line, i.e., $D_{2}$ is an irreducible component of a divisor of $|D_{1}|$. 
    However, since $\dim |D_{1}| = \dim |D_{2}|$, this implies that a general member of $|D_{1}|$ is reducible, which contradicts the assumption. 

    Since $\mathcal{O}_{\mathrm{Rey}(W)}(1)$ and $\omega_{\mathrm{Rey}(W)}(1)$ are numerically equivalent and we only use numerical properties in this argument, the last statement also holds. 
\end{proof}

Let $\mathbb{G}(1,3)$ denote the Grassmannian of lines in $\mathbb{P}^{3}$ or $W$. 
For a line $l \in \mathbb{G}(1,3)$, let $\sigma_{l}$ be the Schubert cycle 
\[
    \sigma_{l} = \{l' \in \mathbb{G}(1,3)\mid l\cap l' \neq \emptyset\}, 
\]
which corresponds to the tangent hyperplane $T_{[l]}\mathbb{G}(1,3)$ in the Pl\"{u}cker embedding $\mathbb{G}(1,3)\hookrightarrow \mathbb{P}^{5}$. 

\begin{lem}\label{lem: union of two lines}
    Let $X$ be the AM double solid associated with an excellent web $W$. 
    Let $(M_{1}^{+})'\subset \overline{M}_{0,1}(X,1)$ (resp.\ $(M_{1}^{-})'\subset \overline{M}_{0,1}(X,1)$) be the component corresponding to $M_{1}^{+} \subset \overline{M}_{0,0}(X,1)$ (resp.\ $M_{1}^{-}\subset \overline{M}_{0,0}(X,1)$). 
    Then $(M_{1}^{+})'\times_{X} (M_{1}^{+})'$, $(M_{1}^{+})'\times_{X} (M_{1}^{-})'$, and $(M_{1}^{-})'\times_{X} (M_{1}^{-})'$ are irreducible. 
\end{lem}
\begin{proof}
    Since the involution $\iota\colon X\rightarrow X$ induces an isomorphism $M_{1}^{+}\cong M_{1}^{-}$, it suffices to prove $(M_{1}^{+})'\times_{X} (M_{1}^{+})'$ and $(M_{1}^{+})'\times_{X} (M_{1}^{-})'$ are irreducible. 
    Moreover, it is enough to show that for general $\ell\in M_{1}^{+}$, the loci 
    \[
        M_{1,\ell}^{+} := \{\ell' \in M_{1}^{+}\mid \ell'\cap \ell \neq \emptyset\}
    \]
    and 
    \[
        M_{1,\iota(\ell)}^{+} := \{\ell' \in M_{1}^{+}\mid \ell'\cap \iota(\ell) \neq \emptyset\}
    \]
    are irreducible. 
    We consider the inverse images of $M_{1,\ell}^{+}$ and $M_{1,\iota(\ell)}^{+}$ by the morphism $\bar{\nu}\colon \mathrm{Rey}(W)\rightarrow M_{1}^{+}$ as in Theorem \ref{thm: AM space of lines}. 
    Note that the union of $M_{1,\ell}^{+}$ and $M_{1,\iota(\ell)}^{+}$ is the inverse image of $\sigma_{f(\ell)} \cap \mathrm{Bit}(\mathcal{D}_{W})$ by the morphism $M_{1}^{+}\hookrightarrow \overline{M}_{0,0}(X,1)\xrightarrow{f_{*}} \mathrm{Bit}(\mathcal{D}_{W})$. 
    Hence the union of $\bar{\nu}^{-1}(M_{1,\ell}^{+})$ and $\bar{\nu}^{-1}(M_{1,\iota(\ell)}^{+})$ equals the inverse image of the hyperplane section $\sigma_{f(\ell)}\cap \mathrm{Bit}(\mathcal{D}_{W})$ by the normalization morphism $\nu\colon \mathrm{Rey}(W)\rightarrow \mathrm{Bit}(\mathcal{D}_{W})$. 
    For a pair $i,j\in \{1,\dots, 10\}$ with $i\neq j$, consider the line $L_{ij} \in \mathrm{Bit}(\mathcal{D}_{W})$ containing the nodes $q_{i}, q_{j}$. 
    The node $q_{i}$ (resp.\ $q_{j}$) corresponds to the union of two plains $P_{i}, P'_{i}$ (resp.\ $P_{j}, P'_{j}$). 
    Then the inverse image $\nu^{-1}(L_{ij})$ consists of four Reye lines 
    \begin{align*}
        \mathfrak{l}_{ij} &= P_{i}\cap P_{j},\\
        \mathfrak{l}_{ij'} &= P_{i}\cap P'_{j},\\
        \mathfrak{l}_{i'j} &= P'_{i}\cap P_{j},\\
        \mathfrak{l}_{i'j'} &= P'_{i}\cap P'_{j}. 
    \end{align*}
    For any smooth quadric $Q \in L_{ij}$, we see that $\mathfrak{l}_{ij}, \mathfrak{l}_{i'j'}$ are in the same ruling of $Q$ and $\mathfrak{l}_{ij'}, \mathfrak{l}_{i'j}$ are in the other ruling. 
    Now we claim that $2\nu^{*}(\sigma_{L_{ij}}\cap \mathrm{Bit}(\mathcal{D}_{W}))$ is linearly equivalent to 
    \[
        \sigma_{\mathfrak{l}_{ij}}\cap \mathrm{Rey}(W) + \sigma_{\mathfrak{l}_{ij'}}\cap \mathrm{Rey}(W) + 
        \sigma_{\mathfrak{l}_{i'j}}\cap \mathrm{Rey}(W) + \sigma_{\mathfrak{l}_{i'j'}}\cap \mathrm{Rey}(W),  
    \]
    and their supports coincide. 
    Let $\mathfrak{l} \in \nu^{-1}(\sigma_{L_{ij}}\cap \mathrm{Bit}(\mathcal{D}_{W}))$.
    Then $\nu(\mathfrak{l})$ intersects $L_{ij}$ at a point $[Q] \in L_{ij}$. 
    Since $Q$ is a smooth quadric or a union of two planes, we see that $\mathfrak{l}$ intersects exactly two of $\mathfrak{l}_{ij}, \mathfrak{l}_{ij'}, \mathfrak{l}_{i'j}, \mathfrak{l}_{i'j'}$.
    Conversely, if $\mathfrak{l}\in \mathrm{Rey}(W)$ intersects some of the four lines, then $\nu(\mathfrak{l})$ intersects $L_{ij}$ since $W$ is base point free. 
    Hence the claim holds. 
    Thus we have either $\nu^{*}(\sigma_{L_{ij}}\cap \mathrm{Bit}(\mathcal{D}_{W})) \in |\mathcal{O}_{\mathrm{Rey}(W)}(2)|$ or $\nu^{*}(\sigma_{L_{ij}}\cap \mathrm{Bit}(\mathcal{D}_{W})) \in |\omega_{\mathrm{Rey}(W)}(2)|$ (In fact, one can check that $\nu^{*}(\sigma_{L_{ij}}\cap \mathrm{Bit}(\mathcal{D}_{W})) \in |\omega_{\mathrm{Rey}(W)}(2)|$ using the fact that $\sigma_{P_{i}}\cap \mathrm{Rey}(W) - \sigma_{P_{i}'}\cap \mathrm{Rey}(W) \sim K_{\mathrm{Rey}(W)}$. However, we omit the proof because, even if $\nu^{*}(\sigma_{L_{ij}}\cap \mathrm{Bit}(\mathcal{D}_{W})) \in |\mathcal{O}_{\mathrm{Rey}(W)}(2)|$ holds, it does not affect the discussion). 

    Let $L \in \mathrm{Bit}(\mathcal{D}_{W})$ be a bitangent line which does not pass through any node. 
    Let $\ell \subset f^{-1}(L)$ be the line on $X$ such that $\ell \in M_{1}^{+}$. 
    Then we claim that $\bar{\nu}^{*}(M_{1,\iota(\ell)}^{+})$ is linearly equivalent to $\sigma_{\nu^{-1}(L)}\cap \mathrm{Rey}(W)$. 
    Indeed, if $\mathfrak{l}\in \sigma_{\nu^{-1}(L)}\cap \mathrm{Rey}(W)$, then $\nu (\mathfrak{l})$ contains a point $[Q]\in L$. 
    If $Q$ is a cone over a smooth conic, then $\bar{\nu}(\mathfrak{l})$ intersects $\iota(\ell)$ at the point $f^{-1}([Q])$. 
    If $Q$ is a smooth quadric, then $\mathfrak{l}$ and $\nu^{-1}(L)$ are in the different rulings of $Q$ since $\mathfrak{l}\cap \nu^{-1}(L) \neq \emptyset$. 
    Recall that the elements of the fiber $f^{-1}([Q])$ are identified with the rulings of $Q$ (Construction \ref{constr: AM modern}). 
    Therefore, the lines $\bar{\nu}(\mathfrak{l})$ and $\ell$ are disjoint, hence $\bar{\nu}(\mathfrak{l}) \cap \iota(\ell) \neq \emptyset$. 
    Conversely, let $\mathfrak{l} \in \bar{\nu}^{-1}(M_{1,\iota(\ell)}^{+})$. 
    Then $\bar{\nu}(\mathfrak{l}) \cap \iota(\ell) \neq \emptyset$. 
    Hence $\nu(\mathfrak{l})$ contains a point $[Q] \in L$. 
    If $Q$ is a cone over a smooth conic, then $\mathfrak{l}$ and $\nu^{-1}(L)$ intersect at the vertex of $Q$. 
    If $Q$ is a smooth quadric, then $\mathfrak{l}$ and $\nu^{-1}(L)$ are in the different rulings of $Q$ since $\bar{\nu}(\mathfrak{l}) \cap \iota(\ell) \neq \emptyset$. 
    Hence $\mathfrak{l}\cap \nu^{-1}(L)\neq \emptyset$, and the claim holds. 
    
    Combining these claims, we have $\bar{\nu}^{*}(M_{1,\ell}^{+})\in |\omega_{\mathrm{Rey}(W)}(1)|$ and $\bar{\nu}^{*}(M_{1,\iota(\ell)}^{+})\in |\mathcal{O}_{\mathrm{Rey}(W)}(1)|$. 
    We now assume that for general $\ell\in M_{1}^{+}$, the locus $M_{1, \ell}^{+}$ is reducible. 
    Then Lemma \ref{lem: Reye reducible div} shows that $\nu^{*}(M_{1, \ell}^{+})$ contains an irreducible component $D$ such that $\dim |D| \le 1$. 
    Since $\omega_{\mathrm{Rey}(W)}(1)\cdot D \le \omega_{\mathrm{Rey}(W)}(1)^{2} = 10$, there are only finitely many possibilities for the linear equivalent class of $D$. 
    Hence there exists a curve $C$ on $\mathrm{Rey}(W)$ such that $C\subset \nu^{-1}(M_{1, \ell}^{+})$ for any general $\ell\in M_{1}^{+}$, which is absurd. 
    Thus, a general $M_{1, \ell}^{+}$ is irreducible. 
    The same argument shows that a general $M_{1, \iota(\ell)}^{+}$ is irreducible, as required. 
\end{proof}
    
\subsection{Higher degree rational curves on AM double solids}\label{subsection: higher deg rational curves on AM}
In this subsection, we prove Theorem \ref{thm: Main thm} and Geometric Manin's Conjecture. 
The following lemma and theorem follow from the results of \cite{Okamura2025Gt}. 
\begin{lem}\label{lem: AM dominant comp}
    Let $X$ be an AM double solid. 
    Then any component of $\overline{M}_{0,0}(X)$ is dominant of the expected dimension.
\end{lem}
\begin{proof}
    By Lemma \ref{lem: AM factorial}, $X$ is a factorial terminal del Pezzo threefold of degree $2$.
    Hence by \cite[Theorem 1.1]{Okamura2025Gt}, $X$ does not have subvarieties $Y$ with $a(Y, -K_{X}|_{Y}) > 1$, which proves the claim by Proposition \ref{prop: non-dom vs higher a-inv}. 
\end{proof}

\begin{thm}[Movable Bend-and-Break]\label{thm: AM MBB}
    Let $X$ be an AM double solid. 
    Let $d \ge 2$ be an integer and let $M \subset \overline{M}_{0,0}(X, d)$ be an irreducible component. 
    Then $M$ contains a stable map $g\colon C_{1} \cup C_{2}\rightarrow X$ such that each restriction $g|_{C_{i}}\colon C_{i}\rightarrow X$ is free. 
\end{thm}
\begin{proof}
    When $d \ge 3$ or $M$ parametrizes double covers of lines, then the claim follows by \cite[Section 5]{Okamura2025Gt}. 
    Hence we may assume that $M$ generically parametrizes smooth $H$-conics. 

    Let $M_{p} \subset M$ be the sublocus parametrizing $H$-conics through a general point $p\in X$. 
    Then by Lemma \ref{lem: free locus}, any rational curve of $H$-degree $\le 2$ through $p$ is free. 
    Hence $M_{p}$ has the expected dimension $2$. 
    Consider the image $N_{p} := f_{*}(M_{p})$. 
    Note that the general curve parametrized by $M_{p}$ maps birationally to a smooth conic on $W\cong \mathbb{P}^{3}$ by $f$. 
    Then the locus in $N_{p}$ parametrizing unions of two lines has dimension $1$, and so does the locus in $M_{p}$ parametrizing unions of two lines. 
    Since any line through $p$ is free, there are only finitely many lines through $p$, all of which do not contain any nodes of $X$. 
    Then by dimension count, one can obtain a stable map $(g\colon C_{1}\cup C_{2}\rightarrow X) \in M_{p}$ such that $p \in g(C_{1})\cap g(C_{2})$. 
    By the generality assumption on $p$, each restriction $g|_{C_{i}}$ is free. 
\end{proof}

We now describe the spaces of conics on AM double solids. 
\begin{thm}\label{thm: AM space of conics}
    Let $X$ be the AM double solid associated with an excellent web $W$. 
    The space $\overline{M}_{0,0}(X, 2)$ of conics on $X$ consists of four irreducible components $R_{2}^{+}, R_{2}^{-}, N_{2}^{+}, N_{2}^{-}$, where $N_{2}^{+}$, $N_{2}^{-}$ parametrize double covers of lines of $M_{1}^{+}$, $M_{1}^{-}$ respectively, and $R_{2}^{+}$, $R_{2}^{-}$ generically parametrize embedded, very free curves. 
\end{thm}
\begin{proof}
    Let $N_{2} \subset \overline{M}_{0,0}(X,2)$ be the locus parametrizing double covers of lines on $X$. 
    By dimension count, each component of $N_{2}$ form a component of $\overline{M}_{0,0}(X,2)$. 
    Since $\overline{M}_{0,0}(X,1)$ consists of two components $M_{1}^{+}, M_{1}^{-}$, we see that $N_{2}$ has also two components $N_{2}^{+}, N_{2}^{-}$ parametrizing double covers of lines of $M_{1}^{+}, M_{1}^{-}$. 

    Let $R_{2}$ be the union of components of $\overline{M}_{0,0}(X,2)$ parametrizing birational maps. 
    By Lemma \ref{lem: AM dominant comp}, any component of $R_{2}$ generically parametrizes free birational maps. 
    Let $M$ be an irreducible component of $R_{2}$. 
    By Theorem \ref{thm: AM MBB}, $M$ contains a union $\ell_{1} \cup \ell_{2}$ of two free lines. 
    Let $\Delta^{++}\subset \overline{M}_{0,0}(X,2)$ be the image of the morphism $(M_{1}^{+})'\times_{X} (M_{1}^{+})'\rightarrow \overline{M}_{0.0}(X,2)$, where $(M_{1}^{+})'\subset \overline{M}_{0,1}(X,1)$ denotes the component corresponding to $M_{1}^{+} \subset \overline{M}_{0,0}(X,1)$, and we define the loci $\Delta^{+-}$ and $\Delta^{--}$ similarly.  
    Then $\ell_{1} \cup \ell_{2}$ is contained in one of $\Delta^{++}, \Delta^{+-}$, or $\Delta^{--}$. 
    By Lemma \ref{lem: union of two lines}, these loci are irreducible. 
    Since $\Delta^{++}$, $\Delta^{+-}$, and $\Delta^{--}$ generically parametrize unions of two free lines, each locus contains smooth points of $\overline{M}_{0,0}(X,2)$. 
    Hence $\Delta^{++}, \Delta^{+-}, \Delta^{--}$ are contained in unique components $R^{++}, R^{+-}, R^{--}$ of $R_{2}$ respectively. 
    For general conics $C^{++}\in R^{++}$, $C^{+-}\in R^{+-}$, $C^{--}\in R^{--}$, let $\tilde{C}^{++}$, $\tilde{C}^{+-}$, $\tilde{C}^{--}$ be the strict transforms of $C^{++}$, $C^{+-}$, $C^{--}$ by the smooth resolution $\phi\colon \tilde{X}\rightarrow X$. 
    Then Lemma \ref{lem: AM alg equiv ve num equiv} shows that $\tilde{C}^{++}$ and $\tilde{C}^{--}$ are algebraically equivalent, but $\tilde{C}^{++}$ and $\tilde{C}^{+-}$ are not. 
    This implies that $R^{+-} \neq R^{++}, R^{--}$. 
    Set $R_{2}^{-} := R^{+-}$. 

    We prove that $R^{++} = R^{--}$. 
    Since $X$ has isolated singularities, the general member $S\in |H|$ is a smooth del Pezzo surface of degree $2$. 
    Let $F(S) \subset \overline{M}_{0,0}(X,1)$ be the set of $56$ lines on $X$ contained in $S$. 
    Now define the map $\mathrm{sign}\colon F(S)\rightarrow \{1, -1\}$ by 
    \[
    \mathrm{sign}(\ell) = 
    \begin{cases}
        1, & \text{if }\ \ell \in M_{1}^{+}, \\
        -1, & \text{if }\ \ell \in M_{1}^{-}.
    \end{cases}
    \]
    This is well-defined since $M_{1}^{+}\cap M_{1}^{-}$ consists of lines through nodes. 
    Moreover, one can see that this map satisfies all the assumptions of Lemma \ref{lem: conic bundles on dP2}. 
    Thus, $S$ admits a conic bundle having singular fibers  $[\ell_{1} \cup \ell_{2}]\in \Delta^{++}$ and $[m_{1}\cup m_{2}]\in \Delta^{--}$. 
    This implies that there is a rational curve $\gamma\colon \mathbb{P}^{1}\rightarrow  R_{2}$ such that $\gamma(0) \in \Delta^{++}\setminus \Delta^{--}$ and $\gamma(1) \in \Delta^{--}\setminus \Delta^{++}$. 
    Then, we must have $R_{2}^{+} := R^{++} = R^{--}$. 
    
    Now the evaluation map $(R_{2}^{+})''\rightarrow X\times X$ (resp.\ $(R_{2}^{-})''\rightarrow X\times X$) is dominant, where $(R_{2}^{+})'', (R_{2}^{-})''\subset \overline{M}_{0,2}(X,2)$ denote the components corresponding to $R_{2}^{+}, R_{2}^{-} \subset \overline{M}_{0,0}(X,2)$. 
    Thus, a general member of $R_{2}^{+}$ (resp.\ $R_{2}^{-}$) is very free, hence also an embedding by \cite[Theorem 3.14]{Kollar1996}, completing the proof. 
\end{proof}

\begin{thm}\label{thm: AM strong MBB}
    Let $X$ be the AM double solid associated with an excellent web $W$. 
    Let $d \ge 2$ and let $M\subset \overline{M}_{0,0}(X, d)$ be a component generically parametrizing birational stable maps. 
    Then $M$ contains unions of free curves $g\colon C_{1} \cup C_{2}\rightarrow X$ and $h\colon D_{1}\cup D_{2}\rightarrow X$ such that $g|_{C_{1}} \in M_{1}^{+}$ and $h|_{D_{1}} \in M_{1}^{-}$.
\end{thm}
\begin{proof}
    The case $d = 2$ is proved in Theorem \ref{thm: AM space of conics}. 
    Suppose $d > 2$. 
    Let $M\subset \overline{M}_{0,0}(X,d)$ be a component generically parametrizing birational maps. 
    By Theorem \ref{thm: AM MBB}, $M$ contains a chain $g\colon C_{1}\cup \dots \cup C_{d}\rightarrow X$ of free lines. 
    Assume that we have $g|_{C_{1}},\dots, g|_{C_{d}} \in M_{1}^{+}$. 
    Since $g|_{C_{1}\cup C_{2}}$ deforms into a union $h'\colon D'_{1}\cup D'_{2}\rightarrow X$ of free lines such that $h'|_{D'_{1}}, h'|_{D'_{2}}\in M_{1}^{-}$, we see that $g$ also deforms into $h\colon D'_{1}\cup D'_{2}\cup C_{3}\cup \dots \cup C_{d}\rightarrow X$ such that $h|_{D'_{1}}, h|_{D'_{2}}\in M_{1}^{-}$ by \cite[Lemma 5.9]{LT2019Compos}. 
    Smoothing the subchain $D'_{2}\cup C_{3}\cup \dots \cup C_{d}\rightarrow X$, we obtain a desired map $h\colon D_{1}\cup D_{2}\rightarrow X$ such that $h|_{D_{1}} \in M_{1}^{-}$. 
\end{proof}

\begin{thm}\label{thm: AM spaces of higher deg curves}
    Let $X$ be the AM double solid associated with an excellent web $W$. 
    For each integer $d\ge 2$, the space $\overline{M}_{0,0}(X, d)$ consists of four irreducible components $R_{d}^{+}, R_{d}^{-}, N_{d}^{+}, N_{d}^{-}$, where $N_{d}^{+}$, $N_{d}^{-}$ parametrize $d$-sheeted covers of lines of $M_{1}^{+}$, $M_{1}^{-}$ respectively, and $R_{d}^{+}$, $R_{d}^{-}$ generically parametrize embedded, very free curves. 
\end{thm}
\begin{proof}
    The case $d = 2$ is proved in Theorem \ref{thm: AM space of conics}. 
    Suppose $d > 2$. 

    By the dimension count, we see that the locus $N_{d}^{+}$ (resp.\ $N_{d}^{-}$) parametrizing $d$-sheeted covers of lines of $M_{1}^{+}$ (resp.\ $M_{1}^{-}$) forms a component of $\overline{M}_{0,0}(X,d)$, and any other component generically parametrizes birational maps. 

    Let $M\subset \overline{M}_{0,0}(X,d)$ be a component generically parametrizing birational maps. 
    By Theorem \ref{thm: AM strong MBB}, $M$ contains stable maps $g\colon C_{1}\cup C_{2}\rightarrow X$ and $h\colon D_{1}\cup D_{2}\rightarrow X$ such that $g|_{C_{1}} \in M_{1}^{+}$ and $h|_{D_{1}}\in M_{1}^{-}$. 
    Now, we consider the loci 
    \begin{align*}
        \Delta_{1,d-1}^{++} &\cong (M_{1}^{+})'\times_{X} (R_{d-1}^{+})', \\
        \Delta_{1,d-1}^{+-} &\cong (M_{1}^{+})'\times_{X} (R_{d-1}^{-})', \\
        \Delta_{1,d-1}^{-+} &\cong (M_{1}^{-})'\times_{X} (R_{d-1}^{+})', \\
        \Delta_{1,d-1}^{--} &\cong (M_{1}^{-})'\times_{X} (R_{d-1}^{-})'
    \end{align*}
      in $\overline{M}_{0,0}(X,d)$, where $(R_{d-1}^{+})' \subset \overline{M}_{0,1}(X, d-1)$ denotes the component corresponding to $R_{d-1}^{+} \subset \overline{M}_{0,0}(X, d-1)$ and so on. 
      By the classification of $a$-covers \cite[Theorem 1.2]{Okamura2025Gt} and Proposition \ref{prop: ev with reducible fib vs a-cov}, we see that the evaluation maps $(R_{d-1}^{+})'\rightarrow X$ and $(R_{d-1}^{-})'\rightarrow X$ have irreducible fibers. 
      Hence the loci $\Delta_{1,d-1}^{++}, \Delta_{1,d-1}^{+-}, \Delta_{1,d-1}^{-+}, \Delta_{1,d-1}^{--}$ are irreducible. 
      Thus each locus is contained in a unique component respectively. 
      However, Theorem \ref{thm: AM strong MBB} implies that $\Delta_{1,d-1}^{++}, \Delta_{1,d-1}^{--}$ are contained in the same component, say $R_{d}^{+}$, and $\Delta_{1,d-1}^{+-}, \Delta_{1,d-1}^{-+}$ are contained in the same component, say $R_{d}^{-}$. 
      Since strict transforms of general curves of $R_{d}^{+}$ and $R_{d}^{-}$ by the smooth resolution $\phi\colon \tilde{X}\rightarrow X$ are not algebraically equivalent, we see that $R_{d}^{+} \neq R_{d}^{-}$. 
      By the argument in Theorem \ref{thm: AM space of conics}, general members of $R_{d}^{+}$, $R_{d}^{-}$ are very free and embedded, which completes the proof.   
\end{proof}

\begin{cor}\label{cor: AM GMC}
    Geometric Manin's Conjecture holds for AM double solids associated with excellent webs.
\end{cor}
\begin{proof}
    The argument in \cite[Theorem 1.4]{Okamura2025Gt} is valid and one can prove that only $R_{d}^{+}$ and $R_{d}^{-}$ for $d\ge 2$ are Manin components. 
    Indeed, by \cite[Theorem 1.1 and Theorem 1.2]{Okamura2025Gt}, any breaking morphism is an $a$-cover of Iitaka dimension $2$, which factors rationally through a family of $H$-lines. 
    Hence $M_{1}^{+}$, $M_{1}^{-}$ and $N_{d}^{+}$, $N_{d}^{-}$ for $d\ge 2$ are accumulating components. 
    Thus, Geometric Manin's Conjecture holds for $X$ since $|\mathrm{Br}_{\mathrm{nr}}(k(X)/k)| = 2$. 
\end{proof}

\bibliography{math}

@article {ArtinMumford1972,
    AUTHOR = {Artin, M. and Mumford, D.},
     TITLE = {Some elementary examples of unirational varieties which are not rational},
   JOURNAL = {Proc. London Math. Soc. (3)},
  FJOURNAL = {Proceedings of the London Mathematical Society. Third Series},
    VOLUME = {25},
      YEAR = {1972},
     PAGES = {75--95},
      ISSN = {0024-6115,1460-244X},
   MRCLASS = {14J10 (14F25)},
  MRNUMBER = {321934},
MRREVIEWER = {M.\ Nagata},
       DOI = {10.1112/plms/s3-25.1.75},
       URL = {https://doi.org/10.1112/plms/s3-25.1.75},
}

@incollection {Beauville2016Luroth,
    AUTHOR = {Beauville, Arnaud},
     TITLE = {The {L}\"{u}roth problem},
 BOOKTITLE = {Rationality problems in algebraic geometry},
    SERIES = {Lecture Notes in Math.},
    VOLUME = {2172},
     PAGES = {1--27},
 PUBLISHER = {Springer, Cham},
      YEAR = {2016},
      ISBN = {978-3-319-46208-0; 978-3-319-46209-7},
   MRCLASS = {14M20 (14E08)},
  MRNUMBER = {3618664},
MRREVIEWER = {Alexander\ Duncan},
}

@article {Beheshti2013,
    AUTHOR = {Beheshti, Roya and Kumar, N. Mohan},
     TITLE = {Spaces of rational curves on complete intersections},
   JOURNAL = {Compos. Math.},
  FJOURNAL = {Compositio Mathematica},
    VOLUME = {149},
      YEAR = {2013},
    NUMBER = {6},
     PAGES = {1041--1060},
      ISSN = {0010-437X,1570-5846},
   MRCLASS = {14M10 (14C05 14M22)},
  MRNUMBER = {3077661},
MRREVIEWER = {Gianluca\ Occhetta},
       DOI = {10.1112/S0010437X12000504},
       URL = {https://doi.org/10.1112/S0010437X12000504},
}

@article {BLRT2022Fano3,
    AUTHOR = {Beheshti, Roya and Lehmann, Brian and Riedl, Eric and
              Tanimoto, Sho},
     TITLE = {Moduli spaces of rational curves on {F}ano threefolds},
   JOURNAL = {Adv. Math.},
  FJOURNAL = {Advances in Mathematics},
    VOLUME = {408},
      YEAR = {2022},
     PAGES = {Paper No. 108557, 60},
      ISSN = {0001-8708,1090-2082},
   MRCLASS = {14H10 (14J30 14J45)},
  MRNUMBER = {4456787},
MRREVIEWER = {Adrian\ Ioan\ Zahariuc},
       DOI = {10.1016/j.aim.2022.108557},
       URL = {https://doi.org/10.1016/j.aim.2022.108557},
}

@article {BlochSrinivas1983algcycle,
    AUTHOR = {Bloch, S. and Srinivas, V.},
     TITLE = {Remarks on correspondences and algebraic cycles},
   JOURNAL = {Amer. J. Math.},
  FJOURNAL = {American Journal of Mathematics},
    VOLUME = {105},
      YEAR = {1983},
    NUMBER = {5},
     PAGES = {1235--1253},
      ISSN = {0002-9327,1080-6377},
   MRCLASS = {14C25 (14C35 18F25)},
  MRNUMBER = {714776},
MRREVIEWER = {K.\ R.\ Coombes},
       DOI = {10.2307/2374341},
       URL = {https://doi.org/10.2307/2374341},
}

@article {Bourqui2012,
    AUTHOR = {Bourqui, David},
     TITLE = {Moduli spaces of curves and {C}ox rings},
   JOURNAL = {Michigan Math. J.},
  FJOURNAL = {Michigan Mathematical Journal},
    VOLUME = {61},
      YEAR = {2012},
    NUMBER = {3},
     PAGES = {593--613},
      ISSN = {0026-2285,1945-2365},
   MRCLASS = {14H10},
  MRNUMBER = {2975264},
MRREVIEWER = {Concettina\ Galati},
       DOI = {10.1307/mmj/1347040261},
       URL = {https://doi.org/10.1307/mmj/1347040261},
}

@article {Bourqui2016,
    AUTHOR = {Bourqui, David},
     TITLE = {Algebraic points, non-anticanonical heights and the {S}everi
              problem on toric varieties},
   JOURNAL = {Proc. Lond. Math. Soc. (3)},
  FJOURNAL = {Proceedings of the London Mathematical Society. Third Series},
    VOLUME = {113},
      YEAR = {2016},
    NUMBER = {4},
     PAGES = {474--514},
      ISSN = {0024-6115,1460-244X},
   MRCLASS = {14M25 (11G35 11G50 14H10)},
  MRNUMBER = {3556489},
MRREVIEWER = {Ariyan\ Javanpeykar},
       DOI = {10.1112/plms/pdw035},
       URL = {https://doi.org/10.1112/plms/pdw035},
}

@article {Browning2017,
    AUTHOR = {Browning, Tim and Vishe, Pankaj},
     TITLE = {Rational curves on smooth hypersurfaces of low degree},
   JOURNAL = {Algebra Number Theory},
  FJOURNAL = {Algebra \& Number Theory},
    VOLUME = {11},
      YEAR = {2017},
    NUMBER = {7},
     PAGES = {1657--1675},
      ISSN = {1937-0652,1944-7833},
   MRCLASS = {14H10 (11P55 14G05 14G25 14J70)},
  MRNUMBER = {3697151},
MRREVIEWER = {R.\ F.\ Lax},
       DOI = {10.2140/ant.2017.11.1657},
       URL = {https://doi.org/10.2140/ant.2017.11.1657},
}

@article{BurkeJovinelly2022Fano3,
      title={Geometric {M}anin's {C}onjecture for {F}ano 3-Folds}, 
      author={Andrew Burke and Eric Jovinelly},
      year={2022},
      eprint={arXiv:2209.05517},
      archivePrefix={arXiv},
      primaryClass={math.AG}
}

@article {Castravet2004,
    AUTHOR = {Castravet, Ana-Maria},
     TITLE = {Rational families of vector bundles on curves},
   JOURNAL = {Internat. J. Math.},
  FJOURNAL = {International Journal of Mathematics},
    VOLUME = {15},
      YEAR = {2004},
    NUMBER = {1},
     PAGES = {13--45},
      ISSN = {0129-167X,1793-6519},
   MRCLASS = {14H60 (14J45 14M20)},
  MRNUMBER = {2039210},
MRREVIEWER = {Gavril\ Farkas},
       DOI = {10.1142/S0129167X0400220X},
       URL = {https://doi.org/10.1142/S0129167X0400220X},
}

@article {Cayley1869/71quartic,
    AUTHOR = {Arthur Cayley},
     TITLE = {A {M}emoir on {Q}uartic {S}urfaces},
   JOURNAL = {Proc. Lond. Math. Soc.},
  FJOURNAL = {Proceedings of the London Mathematical Society},
    VOLUME = {3},
      YEAR = {1869/71},
     PAGES = {19--69},
      ISSN = {0024-6115},
   MRCLASS = {99-04},
  MRNUMBER = {1577187},
       DOI = {10.1112/plms/s1-3.1.19},
       URL = {https://doi.org/10.1112/plms/s1-3.1.19},
}

@article {Clemens1983,
    AUTHOR = {Clemens, C. Herbert},
     TITLE = {Double solids},
   JOURNAL = {Adv. in Math.},
  FJOURNAL = {Advances in Mathematics},
    VOLUME = {47},
      YEAR = {1983},
    NUMBER = {2},
     PAGES = {107--230},
      ISSN = {0001-8708},
   MRCLASS = {14J30 (14J15 14K30)},
  MRNUMBER = {690465},
MRREVIEWER = {David\ Ortland},
       DOI = {10.1016/0001-8708(83)90025-7},
       URL = {https://doi.org/10.1016/0001-8708(83)90025-7},
}

@article {ClemensGriffiths1972IJ,
    AUTHOR = {Clemens, C. Herbert and Griffiths, Phillip A.},
     TITLE = {The intermediate {J}acobian of the cubic threefold},
   JOURNAL = {Ann. of Math. (2)},
  FJOURNAL = {Annals of Mathematics. Second Series},
    VOLUME = {95},
      YEAR = {1972},
     PAGES = {281--356},
      ISSN = {0003-486X},
   MRCLASS = {14J10 (14G13 14J05 14K20 14M20 14N99)},
  MRNUMBER = {302652},
MRREVIEWER = {H.\ Popp},
       DOI = {10.2307/1970801},
       URL = {https://doi.org/10.2307/1970801},
}

@book {CT-Skorobogatov2021,
    AUTHOR = {Colliot-Th\'{e}l\`ene, Jean-Louis and Skorobogatov, Alexei N.},
     TITLE = {The {B}rauer-{G}rothendieck group},
    SERIES = {Ergebnisse der Mathematik und ihrer Grenzgebiete. 3. Folge. A Series of Modern Surveys in Mathematics},
    VOLUME = {71},
 PUBLISHER = {Springer, Cham},
      YEAR = {2021},
     PAGES = {xv+453},
      ISBN = {978-3-030-74247-8; 978-3-030-74248-5},
   MRCLASS = {14F22 (14E08 14G05 14G12 14K05)},
  MRNUMBER = {4304038},
MRREVIEWER = {Thomas\ Benedict\ Williams},
       DOI = {10.1007/978-3-030-74248-5},
       URL = {https://doi.org/10.1007/978-3-030-74248-5},
}

@article {CoskunStarr2009cubic,
    AUTHOR = {Coskun, Izzet and Starr, Jason},
     TITLE = {Rational curves on smooth cubic hypersurfaces},
   JOURNAL = {Int. Math. Res. Not. IMRN},
  FJOURNAL = {International Mathematics Research Notices. IMRN},
      YEAR = {2009},
    NUMBER = {24},
     PAGES = {4626--4641},
      ISSN = {1073-7928,1687-0247},
   MRCLASS = {14H10},
  MRNUMBER = {2564370},
MRREVIEWER = {Scott\ R.\ Nollet},
       DOI = {10.1093/imrn/rnp102},
       URL = {https://doi.org/10.1093/imrn/rnp102},
}

@article {Cossec1983Reye,
    AUTHOR = {Cossec, Fran\c{c}ois R.},
     TITLE = {Reye congruences},
   JOURNAL = {Trans. Amer. Math. Soc.},
  FJOURNAL = {Transactions of the American Mathematical Society},
    VOLUME = {280},
      YEAR = {1983},
    NUMBER = {2},
     PAGES = {737--751},
      ISSN = {0002-9947,1088-6850},
   MRCLASS = {14J28 (14C21 14J10)},
  MRNUMBER = {716848},
MRREVIEWER = {Arnaud\ Beauville},
       DOI = {10.2307/1999644},
       URL = {https://doi.org/10.2307/1999644},
}

@article {Cutkosky1988,
    AUTHOR = {Cutkosky, Steven},
     TITLE = {Elementary contractions of {G}orenstein threefolds},
   JOURNAL = {Math. Ann.},
  FJOURNAL = {Mathematische Annalen},
    VOLUME = {280},
      YEAR = {1988},
    NUMBER = {3},
     PAGES = {521--525},
      ISSN = {0025-5831,1432-1807},
   MRCLASS = {14J30 (13F15 14E35)},
  MRNUMBER = {936328},
MRREVIEWER = {Harry\ D'Souza},
       DOI = {10.1007/BF01456342},
       URL = {https://doi.org/10.1007/BF01456342},
}

@book{DolgachevKondo2024EnriquesII,
  title={Enriques surfaces {II}},
  author={Dolgachev, Igor and Kondo, Shigeyuki},
 publisher={Springer Singapore},
  year={2025},
    isbn={978-981-96-1512-4},
    doi={https://doi.org/10.1007/978-981-96-1513-1}
}

@article {Endrass1999,
    AUTHOR = {Endra\ss , Stephan},
     TITLE = {On the divisor class group of double solids},
   JOURNAL = {Manuscripta Math.},
  FJOURNAL = {Manuscripta Mathematica},
    VOLUME = {99},
      YEAR = {1999},
    NUMBER = {3},
     PAGES = {341--358},
      ISSN = {0025-2611,1432-1785},
   MRCLASS = {14J30 (14E15)},
  MRNUMBER = {1702593},
MRREVIEWER = {S\'{a}ndor\ J.\ Kov\'{a}cs},
       DOI = {10.1007/s002290050177},
       URL = {https://doi.org/10.1007/s002290050177},
}

@article {HaconMcKernan2007RC,
    AUTHOR = {Hacon, Christopher D. and Mckernan, James},
     TITLE = {On {S}hokurov's rational connectedness conjecture},
   JOURNAL = {Duke Math. J.},
  FJOURNAL = {Duke Mathematical Journal},
    VOLUME = {138},
      YEAR = {2007},
    NUMBER = {1},
     PAGES = {119--136},
      ISSN = {0012-7094,1547-7398},
   MRCLASS = {14E30 (14E05 14J45)},
  MRNUMBER = {2309156},
MRREVIEWER = {Mihnea\ Popa},
       DOI = {10.1215/S0012-7094-07-13813-4},
       URL = {https://doi.org/10.1215/S0012-7094-07-13813-4},
}

@article {HRS2004lowdegI,
    AUTHOR = {Harris, Joe and Roth, Mike and Starr, Jason},
     TITLE = {Rational curves on hypersurfaces of low degree},
   JOURNAL = {J. Reine Angew. Math.},
  FJOURNAL = {Journal f\"{u}r die Reine und Angewandte Mathematik. [Crelle's
              Journal]},
    VOLUME = {571},
      YEAR = {2004},
     PAGES = {73--106},
      ISSN = {0075-4102,1435-5345},
   MRCLASS = {14J70 (14J10 14J45)},
  MRNUMBER = {2070144},
MRREVIEWER = {Jaros\l aw\ A.\ Wi\'{s}niewski},
       DOI = {10.1515/crll.2004.045},
       URL = {https://doi.org/10.1515/crll.2004.045},
}

@article {HTT2015Balanced,
    AUTHOR = {Hassett, Brendan and Tanimoto, Sho and Tschinkel, Yuri},
     TITLE = {Balanced line bundles and equivariant compactifications of
              homogeneous spaces},
   JOURNAL = {Int. Math. Res. Not. IMRN},
  FJOURNAL = {International Mathematics Research Notices. IMRN},
      YEAR = {2015},
    NUMBER = {15},
     PAGES = {6375--6410},
      ISSN = {1073-7928,1687-0247},
   MRCLASS = {14G05 (14C20 14E05 14G25)},
  MRNUMBER = {3384482},
MRREVIEWER = {Daniel\ Loughran},
       DOI = {10.1093/imrn/rnu129},
       URL = {https://doi.org/10.1093/imrn/rnu129},
}

@article{HosonoTakagi2015symmetroid,
      title={Geometry of symmetric determinantal loci}, 
      author={Shinobu Hosono and Hiromichi Takagi},
      year={2015},
      eprint={arXiv:1508.01995},
      archivePrefix={arXiv},
      primaryClass={math.AG}
}

@article {IskovskikhManin1971BirAut,
    AUTHOR = {Iskovskikh, V. A. and Manin, Ju.\ I.},
     TITLE = {Three-dimensional quartics and counterexamples to the
              {L}\"uroth problem},
   JOURNAL = {Mat. Sb. (N.S.)},
  FJOURNAL = {Matematicheski\u i\ Sbornik. Novaya Seriya},
    VOLUME = {86(128)},
      YEAR = {1971},
     PAGES = {140--166},
      ISSN = {0368-8666},
   MRCLASS = {14J15},
  MRNUMBER = {291172},
MRREVIEWER = {Peter\ Br\"uckmann},
}

@article{JovinellyOkamura2024coindex3,
      title={Rational Curves on Coindex 3 {F}ano Varieties}, 
      author={Eric Jovinelly and Fumiya Okamura},
      year={2024},
      eprint={arXiv:2409.00834},
      archivePrefix={arXiv},
      primaryClass={math.AG},
      url={https://arxiv.org/abs/2409.00834}, 
}

@article {Kawamata1988,
    AUTHOR = {Kawamata, Yujiro},
     TITLE = {Crepant blowing-up of {$3$}-dimensional canonical
              singularities and its application to degenerations of
              surfaces},
   JOURNAL = {Ann. of Math. (2)},
  FJOURNAL = {Annals of Mathematics. Second Series},
    VOLUME = {127},
      YEAR = {1988},
    NUMBER = {1},
     PAGES = {93--163},
      ISSN = {0003-486X,1939-8980},
   MRCLASS = {14E30 (14B05 14E35 14J10 14J30)},
  MRNUMBER = {924674},
MRREVIEWER = {I.\ Dolgachev},
       DOI = {10.2307/1971417},
       URL = {https://doi.org/10.2307/1971417},
}

@incollection {Kim2001,
    AUTHOR = {Kim, B. and Pandharipande, R.},
     TITLE = {The connectedness of the moduli space of maps to homogeneous
              spaces},
 BOOKTITLE = {Symplectic geometry and mirror symmetry ({S}eoul, 2000)},
     PAGES = {187--201},
 PUBLISHER = {World Sci. Publ., River Edge, NJ},
      YEAR = {2001},
      ISBN = {981-02-4714-1},
   MRCLASS = {14D22 (14L30 14M20)},
  MRNUMBER = {1882330},
MRREVIEWER = {Domenico\ Fiorenza},
       DOI = {10.1142/9789812799821\_0006},
       URL = {https://doi.org/10.1142/9789812799821_0006},
}

@book {Kollar1996,
    AUTHOR = {Koll\'{a}r, J\'{a}nos},
     TITLE = {Rational curves on algebraic varieties},
    SERIES = {Ergebnisse der Mathematik und ihrer Grenzgebiete. 3. Folge. A
              Series of Modern Surveys in Mathematics [Results in
              Mathematics and Related Areas. 3rd Series. A Series of Modern
              Surveys in Mathematics]},
    VOLUME = {32},
 PUBLISHER = {Springer-Verlag, Berlin},
      YEAR = {1996},
     PAGES = {viii+320},
      ISBN = {3-540-60168-6},
   MRCLASS = {14-02 (14C05 14E05 14F17 14J45)},
  MRNUMBER = {1440180},
MRREVIEWER = {Yuri\ G.\ Prokhorov},
       DOI = {10.1007/978-3-662-03276-3},
       URL = {https://doi.org/10.1007/978-3-662-03276-3},
}

@article {KoMiMo1992c,
    AUTHOR = {Koll\'{a}r, J\'{a}nos and Miyaoka, Yoichi and Mori, Shigefumi},
     TITLE = {Rational connectedness and boundedness of {F}ano manifolds},
   JOURNAL = {J. Differential Geom.},
  FJOURNAL = {Journal of Differential Geometry},
    VOLUME = {36},
      YEAR = {1992},
    NUMBER = {3},
     PAGES = {765--779},
      ISSN = {0022-040X,1945-743X},
   MRCLASS = {14J45},
  MRNUMBER = {1189503},
MRREVIEWER = {Yuri\ G.\ Prokhorov},
       URL = {http://projecteuclid.org/euclid.jdg/1214453188},
}

@book {KollarMori1998,
    AUTHOR = {Koll\'ar, J\'anos and Mori, Shigefumi},
     TITLE = {Birational geometry of algebraic varieties},
    SERIES = {Cambridge Tracts in Mathematics},
    VOLUME = {134},
      NOTE = {With the collaboration of C. H. Clemens and A. Corti,
              Translated from the 1998 Japanese original},
 PUBLISHER = {Cambridge University Press, Cambridge},
      YEAR = {1998},
     PAGES = {viii+254},
      ISBN = {0-521-63277-3},
   MRCLASS = {14E30},
  MRNUMBER = {1658959},
MRREVIEWER = {Mark\ Gross},
       DOI = {10.1017/CBO9780511662560},
       URL = {https://doi.org/10.1017/CBO9780511662560},
}

@article {LRT2026nonfree_section,
    AUTHOR = {Lehmann, Brian and Riedl, Eric and Tanimoto, Sho},
     TITLE = {Non-free sections of {F}ano fibrations},
   JOURNAL = {Mem. Amer. Math. Soc.},
  FJOURNAL = {Memoirs of the American Mathematical Society},
    VOLUME = {319},
      YEAR = {2026},
    NUMBER = {1626},
     PAGES = {v+109},
      ISSN = {0065-9266,1947-6221},
      ISBN = {978-1-4704-8151-3; 978-1-4704-8673-0},
   MRCLASS = {14H10 (14E30 14G25 14J45)},
  MRNUMBER = {5075980},
       DOI = {10.1090/memo/1626},
       URL = {https://doi.org/10.1090/memo/1626},
}

@article {LT2019Compos,
    AUTHOR = {Lehmann, Brian and Tanimoto, Sho},
     TITLE = {Geometric {M}anin's conjecture and rational curves},
   JOURNAL = {Compos. Math.},
  FJOURNAL = {Compositio Mathematica},
    VOLUME = {155},
      YEAR = {2019},
    NUMBER = {5},
     PAGES = {833--862},
      ISSN = {0010-437X,1570-5846},
   MRCLASS = {14H10 (14J45)},
  MRNUMBER = {3937701},
MRREVIEWER = {Lidia\ Stoppino},
       DOI = {10.1112/s0010437x19007103},
       URL = {https://doi.org/10.1112/s0010437x19007103},
}

@article {LT2021primeFano,
    AUTHOR = {Lehmann, Brian and Tanimoto, Sho},
     TITLE = {Rational curves on prime {F}ano threefolds of index 1},
   JOURNAL = {J. Algebraic Geom.},
  FJOURNAL = {Journal of Algebraic Geometry},
    VOLUME = {30},
      YEAR = {2021},
    NUMBER = {1},
     PAGES = {151--188},
      ISSN = {1056-3911,1534-7486},
   MRCLASS = {14H10 (14J45 14N35)},
  MRNUMBER = {4233180},
MRREVIEWER = {Zijian\ Zhou},
       DOI = {10.1090/jag/751},
       URL = {https://doi.org/10.1090/jag/751},
}

@article{LT2024dPfibI,
    AUTHOR = {Lehmann, Brian and Tanimoto, Sho},
     TITLE = {Classifying sections of del {P}ezzo fibrations, {I}},
   JOURNAL = {J. Eur. Math. Soc. (JEMS)},
  FJOURNAL = {Journal of the European Mathematical Society (JEMS)},
    VOLUME = {26},
      YEAR = {2024},
    NUMBER = {1},
     PAGES = {289--354},
      ISSN = {1435-9855,1435-9863},
   MRCLASS = {14J26 (14D06 14N35)},
  MRNUMBER = {4705653},
MRREVIEWER = {Caryn\ Werner},
       DOI = {10.4171/jems/1363},
       URL = {https://doi.org/10.4171/jems/1363},
}

@article {MiyaokaMori1986,
    AUTHOR = {Miyaoka, Yoichi and Mori, Shigefumi},
     TITLE = {A numerical criterion for uniruledness},
   JOURNAL = {Ann. of Math. (2)},
  FJOURNAL = {Annals of Mathematics. Second Series},
    VOLUME = {124},
      YEAR = {1986},
    NUMBER = {1},
     PAGES = {65--69},
      ISSN = {0003-486X,1939-8980},
   MRCLASS = {14J40},
  MRNUMBER = {847952},
MRREVIEWER = {C.\ A. M. Peters},
       DOI = {10.2307/1971387},
       URL = {https://doi.org/10.2307/1971387},
}

@article {Mori1979,
    AUTHOR = {Mori, Shigefumi},
     TITLE = {Projective manifolds with ample tangent bundles},
   JOURNAL = {Ann. of Math. (2)},
  FJOURNAL = {Annals of Mathematics. Second Series},
    VOLUME = {110},
      YEAR = {1979},
    NUMBER = {3},
     PAGES = {593--606},
      ISSN = {0003-486X},
   MRCLASS = {14E05 (14M20)},
  MRNUMBER = {554387},
MRREVIEWER = {Luciana\ Picco Botta},
       DOI = {10.2307/1971241},
       URL = {https://doi.org/10.2307/1971241},
}

@article{Okamura2024dP,
    AUTHOR = {Okamura, Fumiya},
     TITLE = {The irreducibility of the spaces of rational curves on del
              {P}ezzo manifolds},
   JOURNAL = {Int. Math. Res. Not. IMRN},
  FJOURNAL = {International Mathematics Research Notices. IMRN},
      YEAR = {2024},
    NUMBER = {12},
     PAGES = {9893--9909},
      ISSN = {1073-7928,1687-0247},
   MRCLASS = {14M20 (14J45)},
  MRNUMBER = {4761783},
       DOI = {10.1093/imrn/rnae080},
       URL = {https://doi.org/10.1093/imrn/rnae080},
}

@article {Okamura2025Gt,
    AUTHOR = {Okamura, Fumiya},
     TITLE = {Rational curves on {F}ano threefolds with {G}orenstein
              terminal singularities},
   JOURNAL = {Eur. J. Math.},
  FJOURNAL = {European Journal of Mathematics},
    VOLUME = {11},
      YEAR = {2025},
    NUMBER = {2},
     PAGES = {Paper No. 38},
      ISSN = {2199-675X,2199-6768},
   MRCLASS = {14H10 (14J45)},
  MRNUMBER = {4918278},
       DOI = {10.1007/s40879-025-00831-y},
       URL = {https://doi.org/10.1007/s40879-025-00831-y},
}

@article {Ottemetal2015quarticspectrahedra,
    AUTHOR = {Ottem, John Christian and Ranestad, Kristian and Sturmfels, Bernd and Vinzant, Cynthia},
     TITLE = {Quartic spectrahedra},
   JOURNAL = {Math. Program.},
  FJOURNAL = {Mathematical Programming},
    VOLUME = {151},
      YEAR = {2015},
    NUMBER = {2},
     PAGES = {585--612},
      ISSN = {0025-5610,1436-4646},
   MRCLASS = {14P10 (14P25 90C22)},
  MRNUMBER = {3348164},
MRREVIEWER = {Jose\ Manuel\ Gamboa},
       DOI = {10.1007/s10107-014-0844-3},
       URL = {https://doi.org/10.1007/s10107-014-0844-3},
}

@article {RavindraSrinivas2006Lefschetz,
    AUTHOR = {Ravindra, G. V. and Srinivas, V.},
     TITLE = {The {G}rothendieck-{L}efschetz theorem for normal projective varieties},
   JOURNAL = {J. Algebraic Geom.},
  FJOURNAL = {Journal of Algebraic Geometry},
    VOLUME = {15},
      YEAR = {2006},
    NUMBER = {3},
     PAGES = {563--590},
      ISSN = {1056-3911,1534-7486},
   MRCLASS = {14C20 (14C22 14C30)},
  MRNUMBER = {2219849},
MRREVIEWER = {Barry\ H.\ Dayton},
       DOI = {10.1090/S1056-3911-05-00421-2},
       URL = {https://doi.org/10.1090/S1056-3911-05-00421-2},
}

@article {RiedlYang2019,
    AUTHOR = {Riedl, Eric and Yang, David},
     TITLE = {Kontsevich spaces of rational curves on {F}ano hypersurfaces},
   JOURNAL = {J. Reine Angew. Math.},
  FJOURNAL = {Journal f\"{u}r die Reine und Angewandte Mathematik. [Crelle's
              Journal]},
    VOLUME = {748},
      YEAR = {2019},
     PAGES = {207--225},
      ISSN = {0075-4102,1435-5345},
   MRCLASS = {14N35 (14J10)},
  MRNUMBER = {3918434},
MRREVIEWER = {Xinli\ Xiao},
       DOI = {10.1515/crelle-2016-0027},
       URL = {https://doi.org/10.1515/crelle-2016-0027},
}

@article {ShimizuTanimoto2022dP1,
    AUTHOR = {Shimizu, Nobuki and Tanimoto, Sho},
     TITLE = {The spaces of rational curves on del {P}ezzo threefolds of degree one},
   JOURNAL = {Eur. J. Math.},
  FJOURNAL = {European Journal of Mathematics},
    VOLUME = {8},
      YEAR = {2022},
    NUMBER = {1},
     PAGES = {291--308},
      ISSN = {2199-675X,2199-6768},
   MRCLASS = {14H10 (14J45 14N35)},
  MRNUMBER = {4389494},
MRREVIEWER = {Adrian\ Ioan\ Zahariuc},
       DOI = {10.1007/s40879-021-00516-2},
       URL = {https://doi.org/10.1007/s40879-021-00516-2},
}

@article{Tanimoto2021GMCintro,
    AUTHOR = {Tanimoto, Sho},
     TITLE = {An introduction to {G}eometric {M}anin's conjecture},
  JOURNAL = {Proceedings of Miyako-no-Seihoku Algebraic Geometry Symposium 2021, “Positivity of tangent bundles and related topics”},
      YEAR = {2021},
     PAGES = {102--118},
  eprint      = {arXiv:2110.06660},
  eprintclass = {math.AG},
  eprinttype  = {arXiv},
}

@book {Testa2005,
    AUTHOR = {Testa, Damiano},
     TITLE = {The {S}everi problem for rational curves on del {P}ezzo surfaces},
      NOTE = {Thesis (Ph.D.)--Massachusetts Institute of Technology},
 PUBLISHER = {ProQuest LLC, Ann Arbor, MI},
      YEAR = {2005},
     PAGES = {(no paging)},
   MRCLASS = {99-05},
  MRNUMBER = {2717265},
       URL = {http://gateway.proquest.com/openurl?url_ver=Z39.88-2004&rft_val_fmt=info:ofi/fmt:kev:mtx:dissertation&res_dat=xri:pqdiss&rft_dat=xri:pqdiss:0808059},
}

@article {Testa2009,
    AUTHOR = {Testa, Damiano},
     TITLE = {The irreducibility of the spaces of rational curves on del
              {P}ezzo surfaces},
   JOURNAL = {J. Algebraic Geom.},
  FJOURNAL = {Journal of Algebraic Geometry},
    VOLUME = {18},
      YEAR = {2009},
    NUMBER = {1},
     PAGES = {37--61},
      ISSN = {1056-3911,1534-7486},
   MRCLASS = {14H10 (14J26)},
  MRNUMBER = {2448278},
MRREVIEWER = {Ulrich\ Derenthal},
       DOI = {10.1090/S1056-3911-08-00484-0},
       URL = {https://doi.org/10.1090/S1056-3911-08-00484-0},
}

@article {Thomsen1998,
    AUTHOR = {Thomsen, Jesper Funch},
     TITLE = {Irreducibility of {$\overline{M}_{0,n}(G/P,\beta)$}},
   JOURNAL = {Internat. J. Math.},
  FJOURNAL = {International Journal of Mathematics},
    VOLUME = {9},
      YEAR = {1998},
    NUMBER = {3},
     PAGES = {367--376},
      ISSN = {0129-167X,1793-6519},
   MRCLASS = {14H10 (14M17 14N10)},
  MRNUMBER = {1625369},
MRREVIEWER = {Alexandre\ I.\ Kabanov},
       DOI = {10.1142/S0129167X98000154},
       URL = {https://doi.org/10.1142/S0129167X98000154},
}

@incollection {Voisin2006IntegralHodge,
    AUTHOR = {Voisin, Claire},
     TITLE = {On integral {H}odge classes on uniruled or {C}alabi-{Y}au
              threefolds},
 BOOKTITLE = {Moduli spaces and arithmetic geometry},
    SERIES = {Adv. Stud. Pure Math.},
    VOLUME = {45},
     PAGES = {43--73},
 PUBLISHER = {Math. Soc. Japan, Tokyo},
      YEAR = {2006},
      ISBN = {978-4-931469-38-9},
   MRCLASS = {14J32 (14C25 14C30)},
  MRNUMBER = {2306166},
MRREVIEWER = {James\ D.\ Lewis},
       DOI = {10.2969/aspm/04510043},
       URL = {https://doi.org/10.2969/aspm/04510043},
}

@article {Voisin2015,
    AUTHOR = {Voisin, Claire},
     TITLE = {Unirational threefolds with no universal codimension {$2$}
              cycle},
   JOURNAL = {Invent. Math.},
  FJOURNAL = {Inventiones Mathematicae},
    VOLUME = {201},
      YEAR = {2015},
    NUMBER = {1},
     PAGES = {207--237},
      ISSN = {0020-9910,1432-1297},
   MRCLASS = {14M20 (14E08 14F45 14J30 14K30)},
  MRNUMBER = {3359052},
MRREVIEWER = {M.\ Kh.\ Gizatullin},
       DOI = {10.1007/s00222-014-0551-y},
       URL = {https://doi.org/10.1007/s00222-014-0551-y},
}
\end{document}